\documentclass[12pt]{amsart}

\usepackage{amssymb}
\usepackage{amsmath} 
\usepackage{amsthm}

\usepackage{amsfonts}
\usepackage{enumerate}
\usepackage[dvipsnames]{xcolor}
\usepackage[shortlabels]{enumitem}
\usepackage[normalem]{ulem}
\usepackage{esint}

\makeatletter
\@namedef{subjclassname@2020}{%
  \textup{2020} Mathematics Subject Classification}
\makeatother

\theoremstyle{plain}
\newtheorem{theorem}{Theorem}[section]
\newtheorem{thm}{Theorem}

\newtheorem{lemma}[theorem]{Lemma}
\newtheorem{corollary}[theorem]{Corollary}
\newtheorem{proposition}[theorem]{Proposition}

\newtheorem*{claim*}{Claim}

\theoremstyle{definition}
\newtheorem*{definition*}{Definition}
\newtheorem{definition}[theorem]{Definition}

\theoremstyle{remark}
\newtheorem{remark}[theorem]{Remark}

\numberwithin{equation}{section}

\def \L {\mathcal{L}}

\def \R {\mathbb{R}}

\def \H {\mathcal{H}}
\def \N {\mathbb{N}}

\def \L {\mathcal{L}}
\def \F {\mathcal{F}}

\def \eps {\varepsilon}

\DeclareMathOperator{\diam}{diam}

\DeclareMathOperator{\Lip}{Lip}

\DeclareMathOperator{\dist}{dist}

\DeclareMathOperator{\Cusp}{Cusp}

\DeclareMathOperator{\hull}{hull}

\newcommand{\res}{\big|}

\begin{document}

%%%%% To ease editing, for IMPAN journals add:

\baselineskip=17pt

%%%%%%%%%%%%%%%%

\title[Lipschitz Functions on Unions and Quotients]{Lipschitz Functions on Unions and Quotients of Metric Spaces}

\author[D. Freeman]{David Freeman}
\address[D. Freeman]{University of Cincinnati Blue Ash College, Blue Ash, OH 45236, USA}
\email{freemadd@ucmail.uc.edu}

\author[C. Gartland]{Chris Gartland}
\address[C. Gartland]{Texas A\&M University, College Station, TX 77843, USA}
\email{cgartland@math.tamu.edu}

\date{\today}

\begin{abstract}
Given a finite collection $\{X_i\}_{i\in I}$ of metric spaces, each of which has finite Nagata dimension and Lipschitz free space isomorphic to $L^1$, we prove that their union has Lipschitz free space isomorphic to $L^1$. The short proof we provide is based on the Pe\l czy\'nski decomposition method. A corollary is a solution to a question of Kaufmann about the union of two planar curves with tangential intersection. A second focus of the paper is on a special case of this result that can be studied using geometric methods. That is, we prove that the Lipschitz free space of a union of finitely many quasiconformal trees is isomorphic to $L^1$. These geometric methods also reveal that any metric quotient of a quasiconformal tree has Lipschitz free space isomorphic to $L^1$. Finally, we analyze Lipschitz light maps on unions and metric quotients of quasiconformal trees in order to prove that the Lipschitz dimension of any such union or quotient is equal to 1.
\end{abstract}



\maketitle

\setcounter{tocdepth}{1}
\tableofcontents

%%%%%%%%%%%%%%%%%%%%%%%%%%%%%%%%%%%%%%%%%
\section{Introduction}
%%%%%%%%%%%%%%%%%%%%%%%%%%%%%%%%%%%%%%%%%

This paper continues the study of Lipschitz functions on quasiconformal (QC) trees as initiated in \cite{FG23} (see Sections \ref{S:prelims} and \ref{S:QC_union} for relevant definitions), expanding the scope of our results to include Lipschitz functions on unions and metric quotients of metric spaces of finite Nagata dimension. We study unions and quotients of QC trees as a special case.

Using the Pe\l czy\'nski decomposition method, we first prove a general result (Theorem \ref{T:General_union}) about the Lipschitz free space of a union of metric spaces with finite Nagata dimension. Via \cite[Theorem C]{FG23}, Theorem \ref{T:General_union} enables us to determine the Lipschitz free space of a finite union of QC trees (Theorem \ref{T:QC_union}). While the proof of Theorem \ref{T:General_union} is conveniently short, it is also rather abstract in nature. Therefore, we also provide a more geometric proof of Theorem \ref{T:QC_union} that does not rely on the Pe\l czy\'nski decomposition method. This geometric proof relies instead upon tools available in the context of doubling metric spaces and provides insight into the metric geometry of QC trees. In particular, our methods shed light on the geometry of branch points in a QC tree (see Theorem \ref{T:uniform} and its use in the proof of Theorems \ref{T:QC_union}, \ref{T:QC_quotient}, and \ref{T:quotient_dim}).

We also study Lipschitz light mappings defined on a union of finitely many QC trees or a metric quotient of a single QC tree. Building on results in \cite{FG23} and \cite{Freeman20}, we are able to prove that the Lipschitz dimension of such a union or quotient is equal to 1. Via 
the results of \cite{CK13} (see also the discussion in \cite[Section 1.2]{David21}), this provides another proof that a union of QC trees bi-Lipschitz embeds into $L^1(Z)$, for some measure space $Z$. Regarding embeddings, we also remind the reader that, by \cite[Theorem 1.2]{DEV23}, any QC tree admits a bi-Lipschitz embedding into some Euclidean space. Thus, by \cite[Theorem 3.2]{LP01}, any union of finitely many QC trees embeds into some Euclidean space. 

In the following subsections we present and discuss our main results in more detail.

%%%%%%%%%%%%%%%%%%%%%%%%%%
\subsection{Lipschitz Free Space Results}
%%%%%%%%%%%%%%%%%%%%%%%%%%

Given a QC tree $T$, one of the main results of \cite{FG23} is that the Lipschitz free space $\mathcal{F}(T)$ is isomorphic to $L^1(Z)$ for some measure space $Z$. This result can be obtained by viewing $T$ as a union of countably many QC arcs $\{\gamma_i\}_{i\in I}$ whose arrangement within $T$ exhibits controlled geometry. In particular, intersecting arcs $\gamma_i$ and $\gamma_j$ exhibit a certain orthogonality property reminiscent of the geometric conditions described in \cite[Proposition 5.1]{Kaufmann14} and \cite[Lemma 3.12]{Weaver18}. In this setting, the space $\F\left(\bigcup_{i\in I}\gamma_i\right)$ is isomorphic to the $\ell^1$-sum $\bigoplus_{i\in I}^1 \F(\gamma_i)$. 

However, in the absence of a controlled geometric relationship between constituent subsets of a union $X=\bigcup_{i\in I}X_i$, methods such as those referenced above cannot be directly applied in order to conclude that the free space on $X$ decomposes into the sum of free spaces on $\{X_i\}_{i\in I}$. This was noted by Kaufmann in the preprint \cite{Kaufmann14} in connection with the metric space 
\[\Cusp:=\{(x,0)\,|\,0\leq x\}\cup\{(x,x^2)\,|\,0\leq x\}\subset\mathbb{R}^2.\]
Kaufmann poses the question of whether or not the Lipschitz free space of a space such as $\Cusp$ is isomorphic to a subspace of $L^1$. One of the main results of this paper implies a positive answer to Kaufmann's question. Indeed, we prove the following.

\begin{thm}\label{T:General_union}
Suppose $X$ is a separable metric space such that $X=X_1\cup X_2$. If, for $i=1,2$, the space $X_i$ has finite Nagata dimension and $\mathcal{F}(X_i)\approx L^1(Z_i)$ for some measure space $Z_i$, then $\mathcal{F}(X)\approx L^1(Z)$ for some measure space $Z$.
\end{thm}

In Corollary~\ref{cor:F(union)=F(M)}, we prove another result of this type about unions of spaces admitting bi-Lipschitz embeddings into $\R^n$ (or more generally, into a self-similar, doubling, bi-Lipschitz homogeneous space, see paragraph preceding Corollary~\ref{cor:F(union)=F(M)} for the definitions). As indicated above, Theorem~\ref{T:General_union} immediately implies the following.

\begin{corollary} $\mathcal{F}(\Cusp)\approx L^1(Z)$ for some measure space $Z$.
\end{corollary}

The proof of Theorem~\ref{T:General_union} relies on the the Pe\l czy\'nski decomposition method (Lemma~\ref{lem:Pdm}) -- a standard tool in Banach space theory used to establish isomorphisms. However, the isomorphism produced by that method is a bit abstract and lacks geometric content. The following theorem is a special case of Theorem~\ref{T:General_union}, and we give a more geometric proof that doesn't invoke Pe\l czy\'nski.

\begin{thm}\label{T:QC_union}
Suppose $X$ is a metric space. If $X=\bigcup_{i\in I}T_i$, where $\{T_i\}_{i\in I}$ is a finite collection of QC trees, there exists a measure space $Z$ such that $\mathcal{F}\left(X\right)\approx L^1(Z)$. 
\end{thm}

We emphasize that Theorem \ref{T:QC_union} assumes nothing about the arrangement of the trees $\{T_i\}_{i\in I}$ in relation to one another. As will be evident in the proof below, this is made possible by the topological fact that any non-degenerate closed and connected subset of a tree is itself a tree. In the course of proving Theorem \ref{T:QC_union} (in a manner independent from Theorem \ref{T:General_union}), we also prove the following result (see Definition~\ref{D:quotient} for the definition of metric quotients).

\begin{thm}\label{T:QC_quotient}
Suppose $T$ is a QC tree. If $M\subset T$ is closed, then $\mathcal{F}(T/M) \approx L^1(Z)$ for some measure space $Z$. 
\end{thm}

The proof of Theorem \ref{T:QC_quotient} relies upon a certain separation property of the branch points in a QC tree, which may be of some interest in its own right (see Theorem \ref{T:uniform}). Finally, in Theorem~\ref{thm:F(quotient)=L1}, we prove a result similar to Theorem~\ref{T:QC_quotient} for quotients of metric spaces by subsets with finite Nagata dimension.

%%%%%%%%%%%%%%%%%%%%%%%%%%%%%%%%%%
\subsection{Lipschitz Dimension Results} 
%%%%%%%%%%%%%%%%%%%%%%%%%%%%%%%%%%

In \cite{FG23} it is also shown that the Lipschitz dimension of any given QC tree $T$ is equal to $1$. As in the study of $\mathcal{F}(T)$, the proof of this result utilizes a geometrically controlled decomposition of $T$ into QC arcs. With this decomposition in hand, Lipschitz light mappings on the constituent QC arcs (provided by \cite[Theorem 2.2]{Freeman20}) can be combined to obtain a Lipschitz light map $f:T\to\mathbb{R}$.

In \cite{David21}, David poses the following question: Given metric spaces $X_1$ and $X_2$, is it true that $\dim_L(X_1\cup X_2)=\max\{\dim_L(X_1),\dim_L(X_2)\}$? The second main result of our paper answers this question in the affirmative when the spaces $X_1$ and $X_2$ are QC trees. Indeed, we prove the following.

\begin{thm}\label{T:Lip_dim_unions}
Suppose $X$ is a metric space. If $X=\bigcup_{i\in I}T_i$, where $\{T_i\}_{i\in I}$ is a finite collection of QC trees, then $\dim_L(X)=1$. 
\end{thm}

We also calculate the Lipschitz dimension of any metric quotient of a QC tree.

\begin{thm}\label{T:quotient_dim}
If $T$ is a QC tree and $M\subset T$ is closed, then $\dim_L(T/M)=1$. 
\end{thm}

Our proof of Theorem \ref{T:quotient_dim} makes use of an apparently new characterization of uniformly disconnected spaces (see Proposition \ref{P:disconnected_quotient}). This may of some interest in its own right.

The organization of the paper is as follows. We provide definitions and notation in Section \ref{S:prelims}. We then study Lipschitz functions on unions of metric spaces and prove Theorem \ref{T:General_union} in Section \ref{S:unions} (via Corollary~\ref{cor:F=L1}). We prove Theorems \ref{T:QC_union} and \ref{T:QC_quotient} (via Theorem~\ref{T:tree_subset}) in Section \ref{S:QC_union}. Finally, we prove Theorems \ref{T:Lip_dim_unions} and \ref{T:quotient_dim} in Section \ref{S:Lip_light}.

%%%%%%%%%%%%%%%%%%%%%%%%%%%%%%%%%%%%%%%%%
\section{Preliminaries}\label{S:prelims}
%%%%%%%%%%%%%%%%%%%%%%%%%%%%%%%%%%%%%%%%%

Here we define a few general concepts that will be frequently referenced in what follows.

\subsection{Lipschitz Functions and Lipschitz Free Spaces}
Given a metric space $(X,d)$ with fixed basepoint $x_0\in X$, we denote by $\Lip_0(X)$ the space of all Lipschitz functions $f:X\to\mathbb{R}$ such that $f(x_0)=0$. Here a function $f:X\to \mathbb{R}$ is said to be \textit{$L$-Lipschitz}, for some $L\geq 1$, provided that, for all $u,v\in X$, we have $|f(u)-f(v)|\leq L\,d(u,v)$. The space $\Lip_0(X)$ is a Banach space when equipped with the norm 
\[\|f\|_{\Lip_0(X)}:=\sup_{u\not=v}\frac{|f(u)-f(v)|}{d(u,v)}.\]
When the space $X$ is understood from the context, we denote this norm by $\|f\|_{\Lip}$. Given a closed subset $A\subset X$ containing $x_0$, we will also work with the subspace 
\[\Lip_A(X):=\{f\in\Lip_0(X)\,|\,f\res_A=0\}.\]

The \textit{Lipschitz free space} of $X$, denoted $\mathcal{F}(X)$, is the canonical Banach space predual of $\Lip_0(X)$ (and the unique predual when $X$ is bounded). The space $\mathcal{F}(X)$ can be realized as the closed linear span of the point evaluation maps $\delta_x\in\Lip_0(X)^*$ defined by $\delta_x(f)=f(x)$, where $x\in X$. The map $\delta: X \to \F(X)$ sending $x$ to $\delta_x$ is an isometric embedding, and it satisfies the following universal property: For every Lipschitz map $f: X \to V$ into a Banach space with $f(x_0) = 0$, there exists a unique bounded linear map $T_f: \F(X) \to V$ such that $T_f \circ \delta = f$, and moreover $\|T_f\| = \Lip(f)$. Throughout the paper, we equip $\Lip_0(X)$ with the weak*-topology coming from the duality $\F(X)^* = \Lip_0(X)$, and we note that this topology on bounded subsets of $\Lip_0(X)$ is the topology of pointwise convergence. It follows easily that $\Lip_A(X)$ is weak*-closed for every closed $A \subset X$.

It is well-known that the isomorphism type of $\mathcal{F}(X)$ is invariant under bi-Lipschitz homeomorphisms of $X$. That is, if $X$ is bi-Lipschitz homeomorphic to $Y$, then $\mathcal{F}(X)$ is linearly isomorphic to $\mathcal{F}(Y)$. See \cite{GK03} or \cite[Chapter 3]{Weaver18} for further background on Lipschitz free spaces, and note that such spaces are also referred to as \textit{Arens-Eells spaces}.

%%%%%%%%%%%%%%%%%%%%%%%%%%%%%%%%%%%%%%%%%
\subsection{Auxiliary Metric Definitions}
%%%%%%%%%%%%%%%%%%%%%%%%%%%%%%%%%%%%%%%%%

Given a metric space $X$, a point $x\in X$, and $r>0$, we write $B(x;r)$ to denote the (closed) metric ball in $X$ of radius $r$ centered at $x$. That is, 
\[B(x;r):=\{y\in X\,|\,d(x,y)\leq r\}.\]
We write $\mathbb{N}$ to denote the set of non-negative integers $\{0,1,2,\dots\}$. We often use the notation $\{x_i\}_{i\in I}$ to denote a sequence indexed by elements of $I\subset \mathbb{N}$. Unless otherwise indicated, we have $I=\{0,1,2,\dots,\max(I)\}$, for $\max(I)<\infty$, or $I=\mathbb{N}$.

\begin{definition}\label{D:net}
Given a metric space $X$, two subsets $A,B\subset X$, and $\varepsilon\in(0,1]$, a subset $N\subset B\setminus A$ is said to be an \textit{$\varepsilon$-Whitney net in $B$ with respect to $A$} provided that, for any pair of points $u \neq v\in N$, 
\[d(u,v)\geq \varepsilon\cdot\max\{d(u,A),d(v,A)\},\]
and $N$ is maximal with respect to this property.
\end{definition}

We note that such an $N$ always exists by Zorn's lemma. Furthermore, $A \cup N$ is a closed set whenever $A$ is closed. The following lemma demonstrates that a Whitney net in $B$ with respect to $A$ accumulates near $A$ if $\dist(A,B)=0$. 

\begin{lemma}\label{L:Whitney}
Suppose $A,B\subset X$ and $N$ is an $\varepsilon$-Whitney net in $B$ with respect to $A$, for some $\varepsilon\in (0,1)$. For any $x\in B\setminus A$, there exists $u\in N$ such that $d(x,u)\leq \varepsilon'\,d(x,A)$, where $\varepsilon':=\varepsilon/(1-\varepsilon)$. 
\end{lemma}

\begin{proof}
Let $x \in A \setminus B$. If $x \in N$, then conclusion holds trivially with $u=x$. Otherwise, the maximality of $N$ implies that there exists $u\in N$ such that 
\[d(x,u)<\varepsilon\cdot \max\{d(x,A),d(u,A)\}.\]
If $d(x,A)\geq d(u,A)$, then we are done. If not, then we have
\[d(u,A)\leq d(x,u)+d(x,A) < \eps \cdot d(u,A)+d(x,A),\]
and so
\[d(u,A)< \frac{1}{1-\eps}d(x,A).\]
It follows that $d(x,u)< \varepsilon'\,d(x,A)$, where $\varepsilon':=\varepsilon/(1-\varepsilon)$.
\end{proof}

We will also make frequent use of the notion of a metric quotient, both in our study of Lipschitz free spaces and of Lipschitz dimension.

\begin{definition}\label{D:quotient}
Given a metric space $(X,d)$ and a closed subset $E\subset X$, the \textit{metric quotient} $(X/E,\rho)$ is defined as the quotient space $X/\sim$, where $x\sim y$ if and only if $x,y\in E$ or $x=y$, and the distance $\rho$ is defined by
\[\rho([a],[b]):=\min\{d(a,b),d(a,E)+d(b,E)\}.\]
Here we write $[a]$ to denote the equivalence class of the point $a\in X$. Given a subset $A\subset X$, we write $[A]:=\{[a]\,|\,a\in A\}\subset X/E$. 
\end{definition}

\begin{remark}\label{R:quotient}
One can easily verify that the quotient map $\pi: X \to X/E$ satisfies the following universal property: a map $f: X/E \to Y$ into a metric space if Lipschitz if and only if $f \circ \pi: X \to Y$ is Lipschitz, and $\Lip(f) = \Lip(f \circ \pi)$. It follows from this that $\Lip_E(X)$ is isometrically weak*-isomorphic to $\Lip_0(X/E)$.
\end{remark}

The proof of the following lemma, while a bit tedious to check, is a straightforward consequence of the relevant definitions. For the sake of brevity, we omit the details.

\begin{lemma}\label{L:double_quotient}
Suppose $Z$ is a metric space and $X\subset Y\subset Z$ are closed. Then the natural identification between $Z/Y$ and  $(Z/X)/(Y/X)$ is an isometry. 
\end{lemma}

%%========================================%%
\section{Lipschitz Functions on Unions and Quotients of Metric Spaces}\label{S:unions}
%%========================================%%

We say that a metric space $X$ has \emph{Nagata dimension $n \in \N$ with constant $\gamma < \infty$} if, for every $s>0$, there exists a cover $\mathcal{C}$ of $X$ such that $\diam(C) \leq \gamma \cdot s$ for every $C \in \mathcal{C}$ and, for every $A \subset X$ with $\diam(A) \leq s$, it holds that $|\{C \in \mathcal{C}: C \cap A \neq \emptyset\}| \leq n+1$. If such an $n$ and $\gamma$ exist, we say that $X$ has \emph{finite Nagata dimension}.

See \cite[Definition~1.1]{LS05} along with references for more on the theory of Nagata dimension. For our purposes, it suffices to record relevant examples that will be used throughout the paper. Recall that a metric space $X$ is \textit{$D$-doubling} if every metric ball in $X$ of radius $r$ can be covered by at most $D$ metric balls in $X$ of radius $r/2$. Recall that an \emph{ultrametric space} $(X,d)$ is one in which $d(x,z) \leq \max\{d(x,y),d(y,z)\}$ for every $x,y,z \in X$.

\begin{lemma} \label{lem:finiteNagatadim}
We will make use of the following facts.
\begin{itemize}
    \item $D$-doubling spaces have Nagata dimension $n$ with constant $\gamma$, where $n$ depends on $D$ and $\gamma$ is universal. \cite[Lemma~2.3]{LS05}
    \item Ultrametric spaces have Nagata dimension 0 with constant 1 (take the cover $\mathcal{C}$ of all balls of radius $s$).
    \item Finite unions of spaces with finite Nagata dimension have finite Nagata dimension \cite[Proposition~2.7]{LS05}.
    \item A space admitting an $L$-bi-Lipschitz embedding into a space with Nagata dimension $n$ with constant $\gamma$ has Nagata dimension $n$ with constant $\gamma'$, where $\gamma'$ depends only on $\gamma$ and $L$ \cite[Lemma~2.1]{LS05}.
\end{itemize}
\end{lemma}

The next lemma highlights the most important consequence (for our purposes) of a space having finite Nagata dimension, and any assumption of finite Nagata dimension in what follows is used only through this lemma.

\begin{lemma} \label{lem:Nagatadecomp}
Suppose $Z$ is a metric space and $X \subset Z$ is a closed subset with finite Nagata dimension. Then there exists a weak*-weak* continuous $L$-bounded linear extension operator $\mathfrak{E}: \Lip_0(X) \to \Lip_0(Z)$, where $L$ depends only on the Nagata dimension $n$ and constant $\gamma$ of $X$. Moreover, $\F(Z) \approx \F(X) \oplus \F(Z/X)$, where the isomorphism constant depends only on $L$.
\end{lemma}

\begin{proof}
Since $X$ has finite Nagata dimension, \cite[Corollary~5.2]{NaorSilb} implies that there exists an $L$-Lipschitz map $Z \to \F(X)$ that restricts to the identity on $X$, where $L$ depends only on $n$ and $\gamma$. Then by the universal property of Lipschitz free spaces, we get an $L$-bounded linear map $\F(Z) \to \F(X)$ that restricts to the identity on $\F(X)$. Dualizing, we get a weak*-weak* continuous $L$-bounded linear extension operator $\mathfrak{E}: \Lip_0(X) \to \Lip_0(Z)$. This proves the first statement. The isomorphism $\F(Z) \approx \F(X) \oplus \F(Z/X)$ then follows from \cite[Lemma 2.2]{Kaufmann15}, with the isomorphism constant depending only on $L$.
\end{proof}

The next two lemmas establish bi-Lipschitz equivalences of spaces involving Whitney nets that will prove to be useful in the proof of Theorem~\ref{thm:F(union)}.

\begin{lemma} \label{lem:F(Whitneynet)}
For every separable metric space $Z$, closed subset $X \subset Z$, and $\varepsilon$-Whitney net $N$ in $Z$ with respect to $X$, the metric quotient $(X\cup N)/X$ is bi-Lipschitz equivalent to an ultrametric space and $\F((X\cup N)/X) \approx \ell^1(S)$ for some countable indexing set $S$.
\end{lemma}

\begin{proof}
Let $[n_1],[n_2] \in (X\cup N)/X$. Then 
\[\rho([n_1],[n_2]) = \min\{d(n_1,n_2),d(n_1,X)+d(n_2,X)\},\] 
and hence by definition of Whitney nets we get
\begin{equation*}
    2\max\{d(n_1,X),d(n_2,X)\} \geq  \rho([n_1],[n_2]) \geq \varepsilon\cdot \max\{d(n_1,X),d(n_2,X)\}.
\end{equation*}
It is obvious that the assignment $([n_1],[n_2]) \mapsto \max\{d(n_1,X),d(n_2,X)\}$ defines an ultrametric on $(X\cup N)/X$, and hence the previous inequalities prove the desired bi-Lipschitz equivalence. The second statement follows from \cite[Theorem~2]{CD16}.
\end{proof}

\begin{lemma} \label{lem:Whitneyiso}
Let $Z$ be a metric space and $X,Y \subset Z$ closed subsets. Given $\varepsilon\in(0,1/2]$ and any $\varepsilon$-Whitney net $N$ in $Y$ with respect to $X$, the identity map $id: \dfrac{Y}{Y\cap (X \cup N)} \to \dfrac{X\cup Y}{X\cup N}$ is a surjective isometry.
\end{lemma}

\begin{proof}
That $id$ is well-defined and 1-Lipschitz is clear. That $id$ is surjective follows from the fact that $N\subset Y$. That $id$ is 1-co-Lipschitz follows from Lemma \ref{L:Whitney} and the assumption that $\varepsilon\leq1/2$.
\end{proof}

We arrive at our first main result on free spaces over unions. It plays a major role in the proof of Theorem~\ref{T:QC_union}.

\begin{theorem} \label{thm:F(union)}
Let $Z$ be a separable metric space and $X,Y \subset Z$ closed subsets. If $|X| = \infty$ and $X,Y$ have finite Nagata dimensions, then there exists a closed subset $F \subset Y$ such that $\F(X \cup Y) \approx \F(X) \oplus \F(Y/F)$.
\end{theorem}

\begin{proof}
Let $N$ be a $\frac{1}{2}$-Whitney net in $Y$ with respect to $X$. Since ultrametric spaces have Nagata dimension 0 and Nagata dimension is preserved under bi-Lipschitz homeomorphisms (see Lemma~\ref{lem:finiteNagatadim}), Lemma~\ref{lem:F(Whitneynet)} implies that $(X\cup N)/X$ has finite Nagata dimension (and thus we may later apply Lemma~\ref{lem:Nagatadecomp}). Since $|X| = \infty$, there exists a Banach space $W$ such that
\begin{equation*}
    \F(X) \approx W \oplus \ell^1
\end{equation*}
by \cite[Theorem~1.1(i)]{CDW16}. By Lemma~\ref{lem:F(Whitneynet)}, there is a countable set $S$ such that
\begin{equation*}
\F((X\cup N)/X) \approx \ell^1(S).
\end{equation*}
Therefore, we have
\begin{equation*}
    \ell^1 \oplus \ell^1(S) \approx \ell^1.
\end{equation*}
Combining these three isomorphisms yields
\begin{align} \label{eq:F(X)+l1}
    \F(X) \oplus \F((X\cup N)/X) &\approx \F(X) \oplus \ell^1(S)\\
    \nonumber &\approx W \oplus \ell^1 \oplus \ell^1(S) \approx W \oplus \ell^1 \approx \F(X).
\end{align}

Then we have
\begin{align*}
    \F(X \cup Y) & \overset{\text{Lem }\ref{lem:Nagatadecomp}}{\approx} \F(X) \oplus \F\left(\frac{X \cup Y}{X}\right) \\
    & \overset{\text{Lems }\ref{lem:Nagatadecomp}, \ref{L:double_quotient}}{\approx} \F(X) \oplus \F\left(\frac{X\cup N}{X}\right) \oplus \F\left(\frac{X \cup Y}{X\cup N}\right) \\
    & \overset{\eqref{eq:F(X)+l1},\text{Lem }\ref{lem:Whitneyiso}}{\approx} \F(X) \oplus \F\left(\frac{Y}{Y \cap (X\cup N)}\right).
\end{align*}
\end{proof}

In order to proceed towards proving Theorem~\ref{T:General_union} (via Corollary~\ref{cor:F=L1}), we require the following version of the Pe\l cy\'nski decomposition method. We include the proof for the convenience of the reader.

\begin{lemma} \label{lem:Pdm}
Suppose that $V$ is a Banach space isomorphic to the countably infinite $\ell^1$-direct sum of itself (e.g. $V = L^1$ or $\ell^1$) and that $W$ is a direct summand of $V$. Then $V \approx W \oplus V$.
\end{lemma}

\begin{proof}
    Let $U$ be a Banach space such that $V \approx U \oplus W$. Then we have
\begin{align*}
	W \oplus V &\approx W \oplus (V \oplus V \oplus V \dots) \\
					&\approx W \oplus ((U \oplus W) \oplus (U \oplus W) \oplus (U \oplus W) \dots) \\
 					&\approx W \oplus (U \oplus (W \oplus U) \oplus (W \oplus U) \oplus \dots) \\
 					&\approx (W \oplus U) \oplus (W \oplus U) \oplus (W \oplus U) \oplus \dots \\
 					&\approx V \oplus V \oplus V \oplus \dots \\
					&\approx V.
\end{align*}
\end{proof}

The next general result is a corollary of Theorem~\ref{thm:F(union)} and Lemma~\ref{lem:Pdm}.

\begin{corollary} \label{cor:F(union)}
Let $Z$ be a separable metric space and $X,Y \subset Z$ closed subsets. Suppose that the following hold.
\begin{itemize}
    \item $X$ and $Y$ have finite Nagata dimensions.
    \item $\F(X)$ is isomorphic to the $\ell^1$-sum $\oplus_{i \in \N}^1 \F(X)$.
    \item $\F(Y)$ is isomorphic to a direct summand of $\F(X)$.
\end{itemize}
Then $\F(X \cup Y) \approx \F(X)$.
\end{corollary}

\begin{proof}
Our assumptions obviously imply that $|X| = \infty$, and thus the hypotheses of Theorem~\ref{thm:F(union)} are met. By that theorem, there exists a closed subset $F \subset Y$ such that $\F(X \cup Y) \approx \F(X) \oplus \F(Y/F)$. By Lemma \ref{lem:Nagatadecomp}, $\F(Y) \approx \F(F) \oplus \F(Y/F)$, showing that $\F(Y/F)$ is a direct summand of $\F(Y)$, and thus a direct summand of $\F(X)$ by assumption. Then by the Pe\l czy\'nski decomposition method (Lemma~\ref{lem:Pdm}),
\begin{equation*}
    \F(X \cup Y) \approx \F(X) \oplus \F(Y/F) \approx \F(X).
\end{equation*}
\end{proof}

We finally arrive at our corollary equivalent to Theorem~\ref{T:General_union}.

\begin{corollary} \label{cor:F=L1}
Let $Z$ be a separable metric space and $X,Y \subset Z$ closed subsets. If $X,Y$ have finite Nagata dimensions and if $\F(X),\F(Y)$ are isomorphic to $L^1$-spaces, then $\F(X \cup Y)$ is isomorphic to an $L^1$-space.
\end{corollary}

The proof will use standard Banach space theoretical facts about $L^1$-spaces. The reader may consult \cite{JL,AO} for references.

\begin{proof}
If one of $|X|,|Y|$ is finite, then the conclusion is obvious. Hence, we may assume $|X|,|Y|=\infty$. Since $Z$ is separable, so are $X,Y$, and thus $\F(X),\F(Y)$ are isomorphic to separable, infinite-dimensional $L^1$-spaces. There are two cases to consider: (i) at least one of $\F(X),\F(Y)$ is isomorphic to $L^1([0,1])$ and (ii) $\F(X) \approx \F(Y) \approx \ell^1$. In case (i), if we assume without loss of generality that $\F(X) \approx L^1([0,1])$, then $\F(Y)$ is a direct summand of $\F(X)$. In case (ii), obviously $\F(Y)$ is also a direct summand of $\F(X)$. In all cases, we may assume without loss of generality that $\F(Y)$ is a direct summand of $\F(X)$ and that $\F(X) \approx \bigoplus_{i \in \N}^1 \F(X)$. Thus, the hypotheses of Corollary~\ref{cor:F(union)} are met, and the conclusion follows.
\end{proof}

Next, we have another corollary identifying the isomorphism type of unions of spaces that admit bi-Lipschitz embeddings into a special class of metric spaces. Let $M$ be a metric space. We say that $M$ is \emph{self-similar} if there exists a constant $R > 1$ and a bijection $f: M \to M$ with $d(f(x),f(y)) = Rd(x,y)$ for all $x,y \in M$ and $f(x_0) = x_0$, and we say that $M$ is \emph{bi-Lipschitz homogeneous}  if for every $x,y \in M$, there exists a bi-Lipschitz homeomorphism $f: M \to M$ with $f(x)=y$. Examples of self-similar, doubling, bi-Lipschitz-homogeneous metric spaces include $\R^n$ and \emph{Carnot groups} (see \cite[p. 7306]{self-similar}, \cite{primer}).

\begin{corollary} \label{cor:F(union)=F(M)}
Let $Z$ be a complete metric space and $M$ a self-similar, doubling, bi-Lipschitz homogeneous metric space (e.g. $M = \R^n$). Suppose $X_0,X_1,\dots X_k \subset Z$ are closed subsets such that each $X_i$ admits a bi-Lipschitz embedding $\phi_i$ into $M$ and $\phi_0(X_0)$ has nonempty interior. Then
$\F(\bigcup_{i=0}^k X_i) \approx \F(M)$.
\end{corollary}

\begin{proof}
The proof is by induction on $k$. The base case $k=0$ holds by \cite[Corollary~5.6]{self-similar}. Suppose that the conclusion holds for some $k \geq 0$. Let $X_0,X_1,\dots X_k,X_{k+1} \subset Z$ be a collection of subsets satisfying the hypotheses of the corollary. Set $X := \bigcup_{i=0}^k X_i$ and $Y := X_{k+1}$. By the inductive hypothesis, $\F(X) \approx \F(M)$, and thus $\F(X) \approx \oplus_{i \in \N}^1 \F(X)$ by \cite[Corollaries~5.5,5.6]{self-similar}. Since $Y$ bi-Lipschitzly embeds into $M$, Lemma~\ref{lem:Nagatadecomp} implies that $\F(Y)$ is a direct summand of $\F(X)$. Hence, the hypotheses of Corollary~\ref{cor:F(union)} are satisfied, and therefore $\F(\bigcup_{i=0}^{k+1} X_i) = \F(X \cup Y) \approx \F(X) \approx \F(M)$.
\end{proof}

We conclude this section with an analogous result for metric quotients. Recall that a metric space $X$ is \emph{purely 1-unrectifiable} if for every subset $A \subset \R$ and Lipschitz map $f: A \to X$, we have $\H^1(f(A)) = 0$, where $\H^1$ denotes 1-dimensional Hausdorff measure. Equivalently, there is no bi-Lipschitz embedding $A \to X$ where $A \subset \R$ has positive Lebesgue measure (see, for example, \cite[Lemma~1.11]{AGPP22} for a discussion of this equivalence).

\begin{theorem} \label{thm:F(quotient)=L1}
Let $X$ be an infinite, separable, complete metric space and $M \subset X$ a closed subset with finite Nagata dimension. Suppose $\F(X)$ is isomorphic to an $L^1$-space. If $X$ is purely 1-unrectifiable, then $\F(X/M) \approx \ell^1(S)$ for some countable set $S$, and if $\H^1(f(A) \setminus M) > 0$ for some Lipschitz $f: \R \supset A \to X$, then $\F(X/M) \approx L^1([0,1])$.
\end{theorem}

\begin{proof}
By Lemma~\ref{lem:Nagatadecomp}, $\F(X) \approx \F(M) \oplus \F(X/M)$. Then either $\F(X) \approx \ell^1$ or $\F(X) \approx L^1([0,1])$. By \cite[Theorem~C]{AGPP22}, the first case happens exactly when $X$ is purely 1-unrectifiable. If $|X/M| < \infty$, then we trivially have $\F(X/M) \approx \ell^1(F)$ for some finite set $F$, so we may assume that $|X/M| = \infty$. In this case, we get that $\F(X/M)$ is isomorphic to $\ell^1$ by \cite[Theorem~15]{AO}. Assume we are in the second case where $\H^1(f(A) \setminus M) > 0$ for some Lipschitz $f: \R \supset A \to X$. Then neither $X$ nor $X/M$ are purely 1-unrectifiable, and so $\F(X) \approx L^1([0,1])$ and $X/M$ contains a bi-Lipschitz copy of a positive measure subset of $\R$, which implies $\F(X/M)$ contains an isomorphic copy of $L^1([0,1])$ by \cite[Corollary~3.4]{Godard10} (see also \cite[Theorem~C]{AGPP22}). In this case, we get that $\F(X/M) \approx L^1([0,1])$ by \cite[page~129]{AO}.
\end{proof}

In light of Theorem~\ref{T:QC_quotient}, we note the following corollary. In the statement, an \textit{$\R$-tree} is a complete metric space $(T,d)$ such that every two points $x,y \in T$ are the endpoints of a unique Jordan arc, and this arc is isometric to the interval $[0,d(x,y)]$ (in other words, it is a geodesic). Thus, $\R$-trees are $1$-bounded turning, but, in contrast to QC trees (see Section \ref{S:QC_union} for these definitions), $\R$-trees need not be doubling, or even proper.

\begin{corollary}
Suppose $T$ is a separable $\R$-tree containing more than one point and $M \subsetneq T$ is a closed subset not equal to $T$. Then $\F(T/M) \approx L^1([0,1])$.
\end{corollary}

\begin{proof}
By \cite[Corollary 3.3]{Godard10}, $\mathcal{F}(T)$ is isomorphic to an $L^1$-space, and by \cite[Theorem 3.2]{LS05}, the Nagata dimension of $T$ is $1$ (in particular, it is finite). Furthermore, $T \setminus M$ is a nonempty open subset, and hence it contains an isometric copy of some sufficiently small interval $(a,b) \subset \R$. The conclusion follows from Theorem~\ref{thm:F(quotient)=L1}.
\end{proof}

%%========================================%%
\section{Lipschitz Functions on Unions and Quotients of QC Trees}\label{S:QC_union}
%%========================================%%

Toward a proof of Theorem \ref{T:QC_union} that does not rely upon the Pe\l czy\'nski decomposition method, we define  relevant terminology and prove auxiliary results over the course of the next few subsections.

\subsection{Metric Geometry of QC Trees}
Given $B\geq 1$, a metric space $X$ is \textit{$B$-bounded turning} provided that any pair of points $u,v\in X$ is contained in some compact and connected set $E\subset X$ such that $\diam(E)\leq B\,d(u,v)$.

A \textit{Jordan arc} is a homeomorphic image of the unit interval $[0,1]$. A metric space $T$ is a \textit{tree} if it is compact, connected, locally connected, and every pair of distinct points in $T$ forms the endpoints of a unique Jordan arc in $T$. Note that, given this definition, any tree is separable. Given a tree $T$, the \textit{leaves} $\mathcal{L}(T)$ are the points $p\in T$ such that $T\setminus\{p\}$ remains connected. The \textit{branch points} $\mathcal{B}(T)$ are the points $p\in T$ such that $T\setminus\{p\}$ consists of at least three connected components. Note that a Jordan arc is a tree with no branch points and exactly two leaves.

We are now ready to define the following key terms.

\begin{definition}\label{D:QC_arc}
A metric space $\gamma$ is a \textit{quasiconformal (QC) arc} if it is a bounded turning and doubling Jordan arc. 
\end{definition}

The reader may object to this terminology in light of the fact that QC arcs are more commonly referred to as \textit{quasi-arcs}. However, for the purposes of this paper, we employ this terminology in order to align with that of the following definition as found in \cite{BM20a}.

\begin{definition}\label{D:QC_tree}
A metric space $T$ is a \textit{quasiconformal (QC) tree} if it is a bounded turning and doubling tree. If $T$ is $C$-bounded turning and $D$-doubling, then we say that $T$ is a $(C,D)$-QC tree.
\end{definition}

The following result follows from \cite[Lemma 2.5]{BM20a}. 

\begin{lemma}\label{L:1bt}
If $(T,d)$ is a $C$-bounded turning tree, then there exists a distance $d'$ such that $(T,d')$ is $1$-bounded turning and $\frac{1}{C}d'\leq d\leq d'$.
\end{lemma}

Next, we embark on a study of the set of branch points $\mathcal{B}(T)$ of a given tree $T$ and the relationship of this set with the set of leaves $\mathcal{L}(T)$. This study culminates in Theorem \ref{T:uniform}, which provides a key ingredient in our proofs of Theorems \ref{T:QC_union} and \ref{T:QC_quotient}.

\begin{lemma}\label{L:components}
Let $T$ be a tree, and $E$ a closed subset of $T$.
\begin{enumerate}[(i)]
    \item{There are countably many components of $T\setminus E$.}
    \item{For any $\varepsilon>0$, there are at most finitely many components of $T\setminus E$ of diameter at least $\varepsilon$.}
    \item{Two points $x,y\in T\setminus E$ are contained in the same component of $T\setminus E$ if and only if $[x,y]\subset T\setminus E$.}
    \item{If $U$ is a component of $T\setminus E$, then the closure $\overline{U}$ is a subtree of $T$ with $\partial\overline{U}\subset\partial U\subset E$.}
    \item{If $E$ is a single point of $T$, then each component of $T\setminus E$ contains a leaf of $T$.}
\end{enumerate}
\end{lemma}

\begin{proof}
These assertions follow from \cite[Lemma 2.3(i)(ii)]{BM20a} and other facts noted on pages 260-261 of \cite{BM20a}.
\end{proof}

\begin{lemma}\label{L:tripods}
Suppose $x_1,x_2,x_3$ are three distinct points in a tree $T$. If $x_i\not\in[x_j,x_k]$ for $i\not\in\{j,k\}$, then the union $[x_1,x_2]\cup[x_1,x_3]\cup[x_2,x_3]$ is a tree containing exactly one branch point given by the singleton contained in $[x_1,x_2]\cap[x_1,x_3]\cap[x_2,x_3]$. 
\end{lemma}

\begin{proof}
This follows from Lemma \ref{L:components} and \cite[Lemma 2.4]{BM20a}. We leave the straightforward details to the reader.
\end{proof}

We note that Lemma \ref{L:tripods} implies that any tree containing no branch points and exactly two leaves is a Jordan arc. 

\begin{lemma}\label{L:locally_connected}
Suppose $T$ is a tree, and $\{x_i\}_{i\in \mathbb{N}}\subset T$ is a sequence of points converging to $x\in T$. Then $\diam([x_i,x])\to0$ as $i\to \infty$. 
\end{lemma}

\begin{proof}
Since $T$ is locally connected, for any $\varepsilon>0$, there exists an open and connected neighborhood $U\subset B(x;\varepsilon)\subset T$ containing $x$. For large enough $i$, we have $x_i\in U$. Since $U$ is open and connected in $T$, it is arc-connected (see \cite[Theorem 8.26]{Nadler92}). It follows from the uniqueness of arcs in $T$ that $[x_i,x]\subset U$, and so the diameter of $[x_i,x]$ is less than $2\varepsilon$. The lemma follows. 
\end{proof}

\begin{definition}\label{D:hull}
Given a tree $T$ and a subset $M\subset \mathcal{L}(T)$, the \textit{convex hull} of $M$ in $T$ is defined as $\hull(M)=\cup_{a,b\in M}[a,b]$.
\end{definition}

\begin{lemma}\label{L:closed_hull}
Given a tree $T$, if $M\subset \mathcal{L}(T)$ is closed in $T$, then $\hull(M)\subset T$ is a subtree of $T$ and $\mathcal{L}(\hull(M))=M$.
\end{lemma}

\begin{proof}
We first prove that $\hull(M)$ is closed. Suppose a sequence of points $\{x_i\}_{i\in \mathbb{N}}\subset \hull(M)$ converges to some point $x\in T$. For each $i\in \mathbb{N}$, there exist points $a_i,b_i\in M$ such that $x_i\in [a_i,b_i]$. Since $M$ is compact, (up to a subsequence) we can assume there exist points $a,b\in M$ such that $a_i\to a$ and $b_i\to b$ as $i\to+\infty$. If $a=b$, then it follows from Lemma \ref{L:locally_connected} that $x_i\to a=x\in \hull(M)$. 

Thus we assume $a\not=b$. If $a_i=a$ and/or $b_i=b$ for all sufficiently large $i$, then our argument simplifies. Thus we assume $a_i\not=a$ and $b_i\not=b$ for all $i$. Via Lemma \ref{L:tripods}, this gives rise to points $a_i',b_i'\in [a,b]$ such that $a_i':=[a,b]\cap[a,a_i]\cap[a_i,b]$ and $b_i':=[a,b]\cap[a,b_i]\cap[b_i,b]$. Lemma \ref{L:locally_connected} implies that the diameters of $[a_i,a_i']\subset[a_i,a]$ and $[b_i,b_i']\subset[b_i,b]$ tend to $0$ as $i\to+\infty$. Note that $[a_i,b_i]\subset[a_i,a_i']\cup[a_i',b_i']\cup [b_i',b_i]$. If there exist arbitrarily large $i$ such that $x_i\not\in[a_i',b_i']$, then it follows that $x\in \{a,b\}\subset\hull(M)$. On the other hand, if $x_i\in [a_i',b_i']\subset[a,b]$ for all sufficiently large $i$, then it follows that $x\in [a,b]\subset\hull(M)$. In either case, we conclude that $x\in\hull(M)$ and so $\hull(M)$ is closed in $T$.

Next, we prove that $\hull(T)$ is (arcwise) connected. Since any closed and connected subset of $T$ is a subtree (see \cite[Lemma 3.3]{BT21}), this will suffice to prove that $\hull(M)$ is a subtree of $T$. Let $x,y\in \hull(T)$. By definition, there exists $\{a_x,b_x,a_y,b_y\}\subset M$ such that $x\in [a_x,b_x]$ and $y\in [a_y,b_y]$. Order $[x,y]$ from $x$ to $y$. Write $x'$ to denote the last point of $[x,y]$ in $[a_x,b_x]$ and write $y'$ to denote the first point of $[x,y]$ in $[a_y,b_y]$. Here we allow for the possibilities that $x'=x$, $y'=y$, or $x'=y'$. In any case, 
\[\hull(M)\supset [a_x,a_y]=[a_x,x']\cup[x',y']\cup[y',a_y],\]
and so 
\[[x,y]=[x,x']\cup[x',y']\cup[y',y]\subset \hull(M).\]

To finish the proof of the lemma, we show that $\mathcal{L}(\hull(M))=M$. First, we note that \cite[Lemma 3.2(2)]{BM22} implies that $M\subset\mathcal{L}(\hull(M))$. Next, given $p\in \mathcal{L}(\hull(M))$, there exist points $a,b\in M$ such that $p\in [a,b]\subset\hull(M)$. If $p\not\in\{a,b\}$, then Lemma \ref{L:components}(iii) implies that $p\not\in[a,b]$. This contradiction implies $p\in M$. It follows that $\mathcal{L}(\hull(M))=M$.
\end{proof}

\begin{lemma}\label{L:sep_pts}
Suppose $T$ is a $1$-bounded turning tree. Given $u,v\in \mathcal{B}(T)$, if $d([u,v],\L(T))>\varepsilon$, then there exist pairwise disjoint arcs $\{\lambda_j\}_{j\in J}=\{[a_j,b_j]\}_{j\in J}$ such that, for each $i,j\in J$, we have
\begin{enumerate}
    \item{$\diam(\lambda_j)=\varepsilon$,}
    \item{$\lambda_j\cap[u,v]=\{a_j\}\in\mathcal{B}(T)$, and}
    \item{$\varepsilon\leq d(b_i,b_j)\leq 2\varepsilon+\diam([u,v])$.}
\end{enumerate}
Furthermore, $|J|=|\mathcal{B}(T)\cap[u,v]|$. 
\end{lemma}

\begin{proof}
Let $\{a_j\}_{j\in J}:=\mathcal{B}(T)\cap [u,v]$. By \cite[Proposition 2.2]{BM20a} and \cite[Theorem 10.23]{Nadler92}, the index set $J$ is countable. For each $j \in J$, choose a component $\Gamma_j$ of $T \setminus \{a_j\}$ that is disjoint from $[u,v]$. Such a component exists because $a_j \in B$ is a branch point. It is easy to verify that the sets $\{[u,v],\Gamma_j\}_{j\in J}$ are pairwise disjoint. Of course, any component of the complement of a point in a compact tree must contain a leaf (Lemma~\ref{L:components}(v)), and thus $\Gamma_j \cap \L(T) \neq \emptyset$ for all $j \in J$. This implies
\begin{equation*}
    \sup_{p \in \Gamma_j} d(a_j,p) \geq d(a_j,\L(T)) >\epsilon,
\end{equation*}
where in the last inequality we've used the fact that $a_j \in [u,v]$ and the assumption that $d([u,v],\L(T))>\varepsilon$. Since each $\Gamma_j$ is connected, we may use the intermediate value theorem and find $b_j \in \Gamma_j$ such that $d(a_j,b_j) = \varepsilon$. Next observe that $[b_j,a_j] \cup [a_j,a_i] \cup [a_i,b_i]$ is the unique arc from $b_j$ to $b_i$ for each $i \neq j \in J$. This is true because the sets $\{[u,v],\Gamma_j\}_{j\in J}$ are pairwise disjoint, $[b_j,a_j) \subset \Gamma_j$, and $[a_j,a_i] \subset [u,v]$. Thus, using the 1-bounded turning property again and the triangle inequality, it holds that
\begin{equation*}
    \varepsilon \leq d(b_i,b_j) \leq 2\varepsilon + d(u,v) = 2\varepsilon + \diam([u,v]),
\end{equation*}
for all $i \neq j \in J$. Setting $\lambda_j:=[a_j,b_j]$, the conclusions of the lemma follow.
\end{proof}

The following lemma is a restatement of \cite[Proposition 3.4]{ACCS01}.

\begin{lemma}\label{L:closed_leaves}
Suppose $T$ is a tree such that $\mathcal{L}(T)$ is closed in $T$. Then the accumulation points of $\mathcal{B}(T)$ are contained in $\mathcal{L}(T)$, and thus $\mathcal{L}(T)\cup\mathcal{B}(T)$ is closed in $T$.
\end{lemma}

By \cite[Proposition 2.2]{BM20a} and \cite[Theorem 10.23]{Nadler92}, the set $\mathcal{B}(T)$ is countable for any tree $T$. This implies that $\mathcal{B}(T)$ is totally disconnected. The following concepts allow us to quantify the disconnectivity.

\begin{definition}
Given $\alpha\in (0,1]$, a finite sequence $\{x_i\}_{i\in I}$ of points in a metric space $X$ is said to be a \emph{relative $\alpha$-chain} if, for each $i<\max(I)<\infty$, we have $d(x_i,x_{i+1})\leq \alpha \cdot d(x_0,x_{\max(I)})$. A relative $\alpha$-chain is \emph{nondegenerate} if its endpoints are distinct. 
\end{definition}

\begin{definition}\label{D:relative_chain}
A metric space $X$ is said to be \textit{$\alpha$-uniformly disconnected} if $X$ contains no nondegenerate relative $\alpha$-chains.
\end{definition}

\begin{remark}
We concede that our use of the adjective \textit{relative} in the above definition is a bit non-standard. However, we include this modifier in order to distinguish Definition \ref{D:relative_chain} from Definition \ref{D:chain} below.
\end{remark}

\begin{remark} \label{rmk:alphachain}
It is immediate from the definition and the triangle inequality that any nondegenerate relative $\alpha$-chain $\{x_i\}_{i\in I}$ in any metric space satisfies $\frac{1}{\alpha} \leq |I|-1$.
\end{remark}

\begin{lemma} \label{lem:alphachainlift}
Let $(X,d)$ be a metric space. Let $B,E \subset X$ with $E$ closed, and $\alpha \in (0,\frac{1}{8}]$. If there exists a nondegenerate relative $\alpha$-chain contained in $[B \cup E] \subset X/E$, then there exists a nondegenerate relative $8\alpha$-chain $\{w_j\}_{j \in J}$ contained in $B \subset X$ with $d(w_0,E)>2d(w_0,w_{\max(J)})$.
\end{lemma}

\begin{proof}
Assume that there exist $[x] \neq [y] \in [B \cup E]$ and $\{[z_i]\}_{i \in I} \subset [B \cup E]$ an $\alpha$-chain from $[x]$ to $[y]$. Without loss of generality, we may assume that $\rho([x],[E]) \geq \rho([y],[E])$ and $[z_i] \neq [z_j]$ for all $i \neq j \in I$, and therefore
\begin{equation} \label{eq:alphachainlift1}
    \rho([x],[E]) \geq \frac{1}{2}\rho([x],[y]).
\end{equation}
From here we consider two cases: the set $\{i \in I: \rho([x],[z_i]) \geq \frac{1}{2}\rho([x],[E])\}$ is empty or is nonempty. Suppose the first case holds. Then the triangle inequality implies $\rho([z_i],[z_j]) < \rho([x],[E])$ and $\rho([z_i],[E]) > \frac{1}{2}\rho([x],[E])$ for all $i,j \in I$, which implies
\begin{equation*}
    \max_{i,j \in I} \rho([z_i],[z_j]) < 2\min_{i \in I} \rho([z_i],[E]).
\end{equation*}
This inequality together with the definition of $\rho$ can be seen to imply that $\rho([z_i],[z_j]) = d(z_i,z_j)$ for all $i,j \in I$. Then the conclusion follows in this case with $\{w_j\}_{j\in J} = \{z_i\}_{i \in I}$.

Now assume that we are in the second case. Set $i_* := \min \{i \in I: \rho([x],[z_i]) \geq \frac{1}{2}\rho([x],[E])\}$.
As before, we have that $\rho([z_i],[z_j]) = d(z_i,z_j)$ for all $i,j < i_*$. Set $\{w_j\}_{j \in J} := \{z_i\}_{i=0}^{i_*-1}$.
It remains to show that $\{w_j\}_{i \in J}$ is a nondegenerate relative $8\alpha$-chain. First we estimate $d(w_0,w_{\max(J)})$:
\begin{align} \label{eq:alphachainlift2}
\nonumber    d(w_0,w_{\max(J)}) &= \rho([x],[z_{i_*-1}]) \\
\nonumber    &\geq \rho([x],[z_{i_*}]) - \rho([z_{i_*-1}],[z_{i_*}]) \\
    &\geq \frac{1}{2}\rho([x],[E]) - \alpha\rho([x],[y]) \\
\nonumber    &\overset{\eqref{eq:alphachainlift1}}{\geq} \frac{1}{4}\rho([x],[y]) - \frac{1}{8}\rho([x],[y]) \\
\nonumber    &= \frac{1}{8}\rho([x],[y]).
\end{align}
Note that this proves the nondegeneracy of $\{w_j\}_{i \in J}$. Then we have, for all $j < \max(J)$,
\begin{align*}
    d(w_j,w_{j+1}) = \rho([z_j],[z_{j+1}]) \leq \alpha \rho([x],[y]) \overset{\eqref{eq:alphachainlift2}}{\leq} 8\alpha d(w_0,w_{\max(J)}).
\end{align*}
\end{proof}

\begin{lemma}\label{L:retract}
Let $T$ be a 1-bounded turning tree, $u,v \in T$, and $\alpha \in (0,1)$. If there exists an $\alpha$-chain $\{x_i\}_{i\in I}$ from $u$ to $v$, then there exists an $\alpha$-chain $\{x'_i\}_{i\in I}$ from $u$ to $v$ contained in $[u,v]$. Furthermore, if $\{x_i\}_{i\in I}\subset\mathcal{B}(T)$, then $\{x'_i\}_{i\in I}\subset \mathcal{B}(T)$.
\end{lemma}

\begin{proof}
Let $(x_i)_{i\in I}$ be an $\alpha$-chain of branch points in $T$ from $u$ to $v$. The idea is to project the chain onto $[u,v]$ via a 1-Lipschitz retraction $g: T \to [u,v]$ with the help of Lemma \ref{L:tripods}. We define the retraction with four cases:
\begin{align*}
    g(x) := \left\{\begin{matrix} x & x \in [u,v] \\
    u & u \in [x,v] \\ v & v \in [u,x] \\ [u,v] \cap [x,v] \cap [u,x] & \text{otherwise.} \end{matrix}\right.
\end{align*}
Note that $g$ is well-defined by Lemma~\ref{L:tripods} by interpreting $[u,v] \cap [x,v] \cap [u,x]$ as the unique point in that singleton set (and not the singleton set itself). Note also that by Lemma~\ref{L:tripods}, $g(x)$ is a branch point whenever $x,u,v$ are branch points. Once we show that $g$ is 1-Lipschitz, the chain $\{g(x_i)\}_{i\in I}$ witnesses the conclusion. Let $x,y \in T$. We check two cases: $[x,y] \cap [u,v]$ is empty or nonempty. In the first case, it holds that $g(x) = g(y)$, and so the 1-Lipschitz condition is trivially satisfied. In the second case, it holds that $[g(x),g(y)] \subset [x,y]$, and so the 1-bounded turning assumption verifies the 1-Lipschitz condition in this case.
\end{proof}

\begin{theorem}\label{T:uniform}
Let $T$ be a $(1,D)$-QC tree with branch set $B$ and leaf set $L$. Then $[B \cup \overline{L}]$ is $\frac{1}{8D^3}$-uniformly disconnected in $T/\overline{L}$.
\end{theorem}

\begin{proof}
Assume that the conclusion is false, so that there exists a nondegenerate relative $\frac{1}{8D^3}$-chain in $[B \cup \overline{L}]$. Then by Lemma \ref{lem:alphachainlift}, there exists a nondegenerate relative $\frac{1}{D^3}$-chain $\{w_j\}_{j \in J}$ contained in $B$ with $d(u,L) > 2d(u,v)$, where $u,v$ are the endpoints of the chain. By the triangle inequality, $d([u,v],L)>d(u,v)$. By Lemma \ref{L:retract}, we may assume that $\{w_j\}_{j\in J} \subset [u,v]$. By Lemma \ref{L:sep_pts}, we obtain a collection of points $\{b_i\}_{i\in I}$ such that, for all $i \neq i'\in I$, we have $d(u,v)\leq d(b_i,b_{i'})\leq 3d(u,v)$. This upper bound implies that $\{b_i\}_{i\in I}$ is contained in a ball of radius $3d(u,v)$. Thus, by the $D$-doubling property, $\{b_i\}_{i\in I}$ is contained in the union of at most $D^3$ balls of radii $\frac{3}{8}d(u,v)$. The lower bound $d(u,v)\leq d(b_i,b_{i'})$ implies that each of these balls contains at most one $b_i$, and thus $|I| \leq D^3$. But by Remark~\ref{rmk:alphachain}, we have $D^3\leq |J|-1\leq |I|-1$, a contradiction.
\end{proof}

\subsection{Proof of Theorem \ref{T:QC_quotient}} Having established Theorem \ref{T:uniform}, we now turn to the study of Lipschitz functions on metric quotients of trees. This study will culminate in the proof of Theorem \ref{T:tree_subset}, which immediately yields a proof of Theorem \ref{T:QC_quotient}.  We begin with the following definition.

\begin{definition}\label{D:sum}
Given a collection $\{(X_i,d_i,p_i)\}_{i\in I}$ of pointed metric spaces, the \textit{sum} $\coprod_{i\in I} (X_i,p_i)$ is the pointed metric space defined by the disjoint union of $\{X_i\}_{i\in I}$ with base point $e$ given by the identification of base points $\{p_i\}_{i\in I}$. Furthermore, given $(a,b)\in X_i\times X_j$, the distance $\sigma$ is defined by
\[\sigma(a,b):=\begin{cases}
                    d_i(a,b)    & \text{ if } i=j\\
                    d_i(a,p_i)+d_j(b,p_j) & \text{ if } i\not=j.
                \end{cases}\]
\end{definition}

\begin{lemma}\label{L:pre_coproduct}
Let $T$ denote a $1$-bounded turning tree and $M\subset T$ a closed subset. The space $T/M$ is $2$-bi-Lipschitz equivalent to $\coprod_{i\in I}(T_i/M_i,[M_i])$. Here $\{T_i\}_{i\in I}$ denotes the closures of the countably many connected components of $T\setminus M$, and, for each $i \in I$, we write $M_i:=T_i\cap M$. 
\end{lemma}

\begin{proof}
We first note that $\{T_i\}_{i\in I}$ is countable by Lemma \ref{L:components}(i). Furthermore, there is a natural identification of $\coprod_{i\in I}(T_i/M_i,[M_i])$ with $T/M$ as sets. Note that $p:=[M]\subset T/M$ corresponds to the based point $e$ of $\coprod_{i\in I}(T_i/M_i,[M_i])$ via this identification. 

Given $x,y\in T/M$, we first suppose $y=p$. In this case, it is easy to see that $\rho(x,p)=\sigma(x,e)$. Therefore, suppose $x,y\in T/M\setminus\{p\}$. Let $T_j$ and $T_k$ denote the components of $T\setminus M$ containing $x$ and $y$, respectively. If $T_j=T_k$, then, since $T$ is $1$-bounded turning, Lemma \ref{L:components}(iii)(iv) imply that 
\begin{align*}
\rho(x,y)&=\min\{d(x,y),d(x,M)+d(y,M)\}\\
    &=\min\{d(x,y),d(x,M_i)+d(y,M_i)\}=\sigma(x,y).
\end{align*}
If $T_j\not=T_k$, then Lemma \ref{L:components} implies that $[x,y]\cap M\not=\emptyset$. In particular, $[x,y]\cap M_j\not=\emptyset\not=[x,y]\cap M_k$. Since $T$ is $1$-bounded turning, 
\begin{equation}\label{E:distances}
d(x,y)\geq \max\{d(x,M_j),d(y,M_k)\}\geq \frac{1}{2}(d(x,M_j)+d(y,M_k)).
\end{equation}
It follows from (\ref{E:distances}) and Lemma \ref{L:components}(iii)(iv) that 
\[\frac{1}{2}\sigma(x,y)\leq \rho(x,y)\leq \sigma(x,y).\]
The desired conclusion follows.
\end{proof}

\begin{definition}\label{D:wreath}
Given $C,D\geq 1$, a \emph{$(C,D)$-QC wreath} is the quotient of a $(C,D$)-QC tree by a two-point subset of $\mathcal{L}(T)$.
\end{definition}

\begin{lemma}\label{L:coproduct}
Let $T$ be a $(1,D)$-QC tree and $M\subset\mathcal{L}(T)$ be closed in $T$. Let $S=\hull(M)$ and $B=\mathcal{B}(S)$. Then $T/(B \cup M)$ is 2-bi-Lipschitz equivalent to a sum $\coprod_{i\in I}(X_i,p_i)$, where $I$ is countable and each $X_i$ is either a $(1,D)$-QC tree or a $(1,D)$-QC wreath.
\end{lemma}

\begin{proof}
It follows from Lemma \ref{L:closed_leaves} and \ref{L:components}(i) that $S\setminus(B\cup M)$ is a collection of countably many pairwise disjoint open arcs whose endpoints are contained in $B\cup M$. We claim that each component of $T\setminus(B\cup M)$ contains at most one component of $S\setminus(B\cup M)$. Indeed, suppose a component of $T\setminus(B\cup M)$ contains components $S_1$ and $S_2$ of $S\setminus(B\cup M)$. Given any $x_1\in S_1$ and $x_2\in S_2$, Lemma \ref{L:components}(iii) implies that $[x_1,x_2]\subset T\setminus(B\cup M)$. Since it is also true that $[x_1,x_2]\subset S\setminus(B\cup M)$, we again apply Lemma \ref{L:components}(iii) to conclude  $S_1=S_2$. 

Let $T'$ denote the closure of a component of $T\setminus (B\cup M)$. We consider two cases. 

\medskip
Case 1: $T'$ contains a component of $S\setminus (B\cup M)$. Denote the closure of this component of $S\setminus(B\cup M)$ by $S'$. By the above paragraph and Lemma \ref{L:components}(iv), there exist points $s_0,s_1\in B\cup M$ such that $S'=[s_0,s_1]$. Again by Lemma \ref{L:components}(iv), the set $T'$ is a subtree of $T$. We also note that the points $\{s_0,s_1\}$ are leaves of $T'$. Indeed, $(B\cup M)\cap T'\subset \partial T'\subset \mathcal{L}(T')$.

We claim that $T'\cap(B\cup M)=\{s_0,s_1\}$. To see this, suppose there exists a point $x_0\in T'\cap (B\cup M)\subset\mathcal{L}(T')$ such that $x_0$ is not an endpoint of $S'$. Since $s_0$, $s_1$, and $x_0$ are leaves of $T'$, we have $s_0\not\in[x_0,s_1]$, $s_1\not\in[x_0,s_0]$, and $x_0\not\in[s_0,s_1]$. By Lemma \ref{L:tripods}, the open arc $(s_0,s_1)$ contains a branch point of $S$. This contradicts the fact that $(s_0,s_1)\cap B=\emptyset$. Therefore, we verify our claim that $T'\cap(B\cup M)=\{s_0,s_1\}=:P'$. We conclude that the image of $T'$ in the quotient $T/(B\cup M)$ is the $(1,D)$-QC wreath $T'/P'$.

\medskip
Case 2: $T'$ does not contain any component of $S\setminus(B\cup M)$. It follows that: 
\begin{equation}\label{E:non_intersecting}
T' \textrm{ does not intersect any component of } S\setminus(B\cup M).
\end{equation}
We claim that $T'$ intersects $B\cup M$ in exactly one point. Since $T$ is connected, this intersection is non-empty. If $T'\cap (B\cup M)$ contains points $x_0\not=x_1$, then  $[x_0,x_1]\subset S\cap T'$. Since not every point of $(x_0,x_1)$ can be a branch point of S (see \cite[Theorem 10.23]{Nadler92}), and $(x_0,x_1)\cap M=\emptyset$, it follows that $(x_0,x_1)$ intersects a component of $S\setminus(B\cup M)$. This contradicts (\ref{E:non_intersecting}). Therefore, we verify our claim that $P':=T'\cap(B\cup M)$ contains exactly one point. We conclude that the image of $T'$ in $T/(B\cup M)$ is the $(1,D)$-QC tree $T'/P'$. 

\medskip
Write $\{T_i\}_{i\in I}$ to denote the countable collection of the closures of connected components of $T\setminus (B\cup M)$, and, for each $i\in I$, define $P_i:=T_i\cap (B\cup M)$. By Lemma \ref{L:pre_coproduct}, the quotient $T/(B\cup M)$ is $2$-bi-Lipschitz equivalent to $\coprod_{i\in I}(T_i/P_i,[P_i])$. Cases 1 and 2 above confirm that each $T_i/P_i$ is either a $(1,D)$-QC wreath or $(1,D)$-QC tree.
\end{proof}

By taking pre-duals and appealing to Lemma \ref{L:1bt}, the following theorem provides a proof of Theorem \ref{T:QC_quotient}.
 
\begin{theorem}\label{T:tree_subset}
Suppose $T$ is a $(1,D)$-QC tree. If $M\subset T$ is closed, then $\Lip_0(T/M)$ is weak*-isomorphic to $L^\infty(Z)$ for some measure space $Z$. 
\end{theorem}

\begin{proof}
By Lemma \ref{L:components}, the closures of the (countably many) components of $T\setminus M$ are $(1,D)$-QC trees $\{T_i\}_{i\in I}$. Let $M_i$ denote the subset of the leaves of $T_i$ that are contained in $M$. Note that $M_i$ is closed in $T_i$, since $M$ is closed and $M_i=T_i\cap M$. Lemma \ref{L:pre_coproduct} implies that $T/M$ is $2$-bi-Lipschitz equivalent to the sum $\coprod_{i\in I}T_i/M_i$. It then follows from \cite[Proposition 2.8(b)]{Weaver18} that 
\begin{equation}\label{E:first}
\Lip_0(T/M)\approx \bigoplus_{i\in I}\Lip_0(T_i/M_i),
\end{equation}
where the isomorphism constant is absolute (throughout this proof, ``$\approx$" denotes a weak*-weak* continuous isomorphism between dual Banach spaces).

Let $S_i=\hull(M_i)$ in $T_i$
and write $B_i$ to denote the branch points of $S_i$. Observe that, by Theorem \ref{T:uniform}, for each $i\in I$, the space $[B_i\cup M_i]\subset S_i/M_i\subset T_i/M_i$ is $\frac{1}{8D^3}$-uniformly disconnected. Then by \cite[p.\,161]{DS97}, the space $[B_i\cup M_i]$ is $8D^3$-bi-Lipschitz equivalent to an ultrametric space, and hence has Nagata dimension 0 with constant depending only on $D$ (and thus we may apply Lemma~\ref{lem:Nagatadecomp}). By Lemmas \ref{L:closed_hull}, \ref{L:closed_leaves}, \ref{L:double_quotient}, \ref{lem:Nagatadecomp}, and Remark \ref{R:quotient}, for each $i\in I$, we have 
\begin{align}\label{E:second}
\Lip_0(T_i/M_i)&\approx \Lip_0([B_i\cup M_i])\oplus\Lip_{[B_i\cup M_i]}(T_i/M_i)\\
\nonumber &=\Lip_0([B_i\cup M_i])\oplus\Lip_0(T_i/(B_i\cup M_i)).
\end{align}
Here the isomorphism constant depends only on the doubling constant $D$.

Lemma \ref{L:coproduct} tells us that $T_i/(B_i\cup M_i)$ is $2$-bi-Lipschitz equivalent to the sum $\coprod_{j\in J_i} X_{i,j}$, where each $X_{i,j}$ is either a $(1,D)$-tree or a $(1,D)$-wreath. By \cite[Proposition 2.8(b)]{Weaver18}, we conclude that 
\begin{equation}\label{E:quotient_sum}
\Lip_0(T_i/(B_i\cup M_i))\approx \bigoplus_{j\in J_i} \Lip_0(X_{i,j}).
\end{equation}
Again the isomorphism constant is absolute. 

Suppose $X_{i,j}$ is a wreath. By definition, for each $i \in I$ and $j \in J_i$, there exists a $(1,D)$-QC tree $T_{i,j}$ and a two-point subset $P_{i,j} \subset \mathcal{L}(T_{i,j})$ such that $X_{i,j} = T_{i,j}/P_{i,j}$. We may assume that $P_{i,j}$ contains the basepoint $x_{i,j}\in T_{i,j}$ at which all Lipschitz functions in $\Lip_0(T_{i,j})$ are zero. Hence, $\Lip_0(X_{i,j})$ is a weak*-closed subspace of $\Lip_0(T_{i,j})$ with codimension $1$. By \cite[Theorem~C]{FG23}, $\Lip_0(T_{i,j})\approx L^\infty(Z'_{i,j})$ for some measure space $Z'_{i,j}$ with isomorphism constant depending only on $D$. It follows that $\Lip_0(X_{i,j})\approx L^\infty(Z_{i,j})$ for some measure space $Z_{i,j}$ with isomorphism constant depending only on $D$. 

Suppose now that $X_{i,j}$ is a tree. In this case, then, again referencing \cite[Theorem C]{FG23}, we conclude that $\Lip_0(X_{i,j})\approx L^\infty(Z_{i,j})$ for some measure space $Z_{i,j}$ with isomorphism constant depending only on $D$.

In either case, it now follows from (\ref{E:quotient_sum}) that, for some measure space $Z_i$, we have
\begin{equation}\label{E:third}
\Lip_0(T_i/(B_i\cup M_i))\approx \bigoplus_{j\in J_i}L^\infty(Z_{i,j})\approx L^\infty(Z_i),
\end{equation}
where the isomorphism constants depend only on $D$.

As observed earlier, each space $[B_i\cup M_i]$ is $8D^3$-bi-Lipschitz equivalent to an ultrametric space. It then follows from \cite[Theorem 2]{CD16} that
\begin{equation}\label{E:fourth}
\Lip_0([B_i\cup M_i])\approx \ell^\infty(S_i),
\end{equation}
for some countable index set $S_i$. Here the isomorphism constant again depends only on $D$.

Since all relevant isomorphism constants depend only on the doubling constant $D$, we conclude that
\begin{align*}
\Lip_0(T/M)\overset{(\ref{E:first})}{\approx}&\bigoplus_{i\in I}\Lip_0(T_i/M_i)\\
\overset{(\ref{E:second}) \text{, Lem } \ref{L:double_quotient}}{\approx}&\bigoplus_{i\in I}\left(\Lip_0([B_i\cup M_i])\oplus\Lip_0(T_i/(B_i\cup M_i))\right)\\
\overset{(\ref{E:third})}{\approx}&\bigoplus_{i\in I}\left(\Lip_0([B_i\cup M_i])\oplus L^\infty(Z_i)\right)\\
\overset{(\ref{E:fourth})}{\approx}&\bigoplus_{i\in I}\left(\ell^\infty(S_i)\oplus L^\infty(Z_i)\right)\approx L^\infty(Z),
\end{align*}
for some measure space $Z$.
\end{proof}

%%%%%%%%%%%%%%%%%%%%%%%%%%%%%%%%%%%%%%%%%%%%%%%%%%
\subsection{Proof of Theorem \ref{T:QC_union}}
%%%%%%%%%%%%%%%%%%%%%%%%%%%%%%%%%%%%%%%%%%%%%%%%%%

\begin{proof}
Suppose the indexing set is $I = \{0,1, \dots k\}$. We will prove the theorem by induction on $k$. For the base case, note that the conclusion holds for $T_0$ by \cite[Theorem~C]{FG23}. Assume that the conclusion holds for the union $X := \bigcup_{i<k}T_i$. Thus the sets $X, Y = T_k$, satisfy the assumptions of Theorem~\ref{thm:F(union)}, and we conclude that 
\[\F\left(\bigcup_{i\leq k}T_i\right) \approx \F(X)\oplus\F(T_k/M)\]
for some closed $M \subset T_k$. The inductive hypothesis and Theorem~\ref{T:QC_quotient} imply that
\[\F\left(\bigcup_{i\leq k}T_i\right)\approx L^1(Z_{k-1})\oplus L^1(Z_k')\approx L^1(Z_k)\]
for some measure spaces $Z_{k-1}$, $Z_k'$ and $Z_k$. This completes the inductive step.
\end{proof}

%%=====================================================================%%
\section{Lipschitz Light Maps on Unions and Quotients of QC Trees}\label{S:Lip_light}
%%=====================================================================%%

\subsection{Lipschitz Light Maps and Lipschitz Dimension}

We begin with some terminology that underlies the concept of a Lipschitz light map.

\begin{definition}\label{D:chain}
Given $\delta>0$, we say that a finite sequence $\{u_i\}_{i\in I}$ is a \textit{$\delta$-chain} provided that, for each $i<\max(I)$, we have $d(x_i,x_{i+1})\leq \delta$.
\end{definition}

A subset $U$ of a metric space $X$ is \textit{$\delta$-connected} if every pair of points in $U$ is contained in a $\delta$-chain in $U$. A \textit{$\delta$-component} of $X$ is a maximal $\delta$-connected subset of $X$. 

\begin{definition}\label{D:LL}
A map $f:X\to Y$ between metric spaces is \textit{Lipschitz light} if there exist constants $L,Q>0$ such that \begin{enumerate}
    \item{$f$ is $L$-Lipschitz, and}
    \item{for every $r>0$ and $E\subset Y$ such that $\diam(E)\leq r$, the $r$-components of $f^{-1}(E)$ have diameter at most $Qr$.}
\end{enumerate}
We say that such an $f$ is $L$-Lipschitz and $Q$-light. A collection of maps $\{f_i\}_{i\in I}$ is said to be \textit{uniformly} Lipschitz light if there exist $L,Q>0$ such that, for every $i\in I$, the map $f_i$ is $L$-Lipschitz and $Q$-light.
\end{definition}

\begin{remark}\label{R:LipLight}
In \cite[Section 1.4]{David21}, David points out that the above definition of a Lipschitz light map is equivalent to the following for maps into Euclidean space: There exist $L,Q>0$ such that $f$ is $L$-Lipschitz, and, for every bounded subset $E\subset \mathbb{R}^d$, the $\diam(E)$-components of $f^{-1}(E)$ have diameter at most $Q\cdot\diam(E)$.
\end{remark}

It is easy to check that $f_2 \circ f_1$ is $L_1L_2$-Lipschitz $L_1Q_1Q_2$-light whenever $f_1$ is $L_1$-Lipschitz $Q_1$-light and $f_2$ is $L_2$-Lipschitz $Q_2$-light. We will use this fact throughout.

\begin{definition}\label{D:Lip_light}
A metric space $X$ has \textit{Lipschitz dimension} at most $n$ if there exists a Lipschitz light map $f:X\to \mathbb{R}^n$. The Lipschitz dimension of $X$ is the infimal such $n$. 
\end{definition}

We take an infimum instead of a minimum in the above definition because the Lipschitz dimension of a space may be infinite (see \cite{David21}). A few reasons that Lipschitz dimension is of theoretical significance are provided by the embedding results for spaces of Lipschitz dimension at most 1 contained in \cite{CK13} and certain non-embedding results for spaces of infinite Lipschitz dimension contained in \cite{David21}.

\subsection{Proof of Theorem \ref{T:Lip_dim_unions}}\label{s:Lip_dim}

We begin this section with two general lemmas on Lipschitz light maps that will be used in the proof of Theorem~\ref{T:Lip_dim_unions}.

\begin{lemma}\label{L:light_projection}
Suppose $X$ is a metric space, $A,B \subset X$, $\eps \in (0,1]$, and $N$ is an $\eps$-Whitney net in $B$ with respect to $A$. Let $c \in (1,\infty)$, and let $\pi: A \cup N \to A$ be any map satisfying $d(u,\pi(u)) \leq c \cdot d(u,A)$ for all $u \in A \cup N$. Then $\pi$ is a $(2c+\varepsilon)/\varepsilon$-Lipschitz $(2c+\varepsilon)/\varepsilon$-light map.
\end{lemma}

\begin{proof}
First we note that, for all $x,y \in A \cup N$, regardless of whether $x,y$ belong to $A$ or $N$, the Whitney net inequality
\begin{equation} \label{eq:Whitneyineq}
    d(x,y) \geq \eps \cdot \max\{d(x,A),d(y,A)\}
\end{equation}
still holds. Then by this and the definition of $\pi$, we have
\[d(\pi(x),\pi(y))\leq d(\pi(x),x)+d(x,y)+d(\pi(y),y)\leq \left(\frac{2c}{\varepsilon}+1\right)d(x,y).\]
Therefore, $\pi$ is $(2c+\varepsilon)/\varepsilon$-Lipschitz.

To see that $\pi$ is light, fix any $\delta>0$ and choose any $E \subset A$ such that $\diam(E) \leq \delta$. Let $\{z_i\}_{i\in I}$ denote any $\delta$-chain in $\pi^{-1}(E)\subset A \cup N$. We observe, for any $1\leq i \leq \max(I)$, that
\[\delta \geq d(z_{i-1},z_i) \overset{\eqref{eq:Whitneyineq}}{\geq} \varepsilon\cdot\max\{d(z_{i-1},A),d(z_i,A)\}.\]
Therefore, every point of $\{z_i\}_{i\in I}$ is within distance $\delta/\varepsilon$ of $A$. By the definition of $\pi$, for every $i,j\in I$, we therefore obtain
\begin{align*}
d(z_i,z_j)&\leq d(z_i,\pi(z_i))+d(\pi(z_i),\pi(z_j))+d(z_j,\pi(z_j))\\
&\leq \frac{2c\delta}{\varepsilon}+d(\pi(z_i),\pi(z_j))\\
&\leq \left(\frac{2c}{\varepsilon}+1\right)\delta,
\end{align*}
where the final inequality follows from the fact that $\{\pi(z_i),\pi(z_j)\}\subset E$ and $\diam(E)\leq \delta$. This shows that $\pi$ is $(2c+\varepsilon)/\varepsilon$-light.
\end{proof}

\begin{lemma}\label{L:two_piece}
Let $f: X \to Y$ be a map between metric spaces. If there exist $Q<\infty$ and subsets $A,B \subset X$ such that $X = A \cup B$, $f\res_A$ is $Q$-light, and $f\res_B$ is $Q$-light, then $f$ is $(2Q(Q+2)+1)$-light.
\end{lemma}

\begin{proof}
Let $Q,A,B$ be as above. Without loss of generality, we may assume that $A\cap B = \emptyset$. Let $\delta > 0$, and choose $E \subset Y$ with $\diam(E) \leq \delta$. Let $\{x_i\}_{i\in I}$ be a $\delta$-chain in $f^{-1}(E)$. We need to show that $\diam(\{x_i\}_{i\in I}) \leq (2Q(Q+2)+1)\delta$. We may assume, for our purposes, that $x_{0},x_{\max(I)}\in \{x_i\}_{i\in I}$ are such that $\diam(\{x_i\}_{i\in I})=d(x_{0},x_{\max(I)})$. Partition $I$ into $I = I_A \sqcup I_B$ such that $\{x_i\}_{i\in I_A} \subset A$ and $\{x_i\}_{i\in I_B} \subset B$. Without loss of generality, we may assume that
\begin{equation}\label{eq:diamA>diamB}
    \diam(\{x_i\}_{i\in I_A}) \geq \diam(\{x_i\}_{i\in I_B}).
\end{equation}
Obviously, this implies $I_A \neq\emptyset$. If $I_B = \emptyset$, then we have $\diam(\{x_i\}_{i\in I}) = \diam(\{x_i\}_{i\in I_A}) \leq Q\delta \leq (2Q(Q+2)+1)\delta$ from the fact that $f\res_A$ is $Q$-light, and we are done. We may assume, then, that $I_B \neq \emptyset$. Since $I_A,I_B\neq \emptyset$, there must exist consecutive points $x_{i'},x_{i'+1}$ such that $x_{i'} \in A$ and $x_{i'+1} \in B$, or vice versa. If $x_{0},x_{\max(I)}\in A$ or $x_{0},x_{\max(I)}\in B$, then (\ref{eq:diamA>diamB}) implies that 
\[\diam(\{x_i\}_{i\in I})=d(x_{0},x_{\max(I)})\leq \diam(\{x_i\}_{i\in I_A}).\]
If $x_{0}\in A$ and $x_{\max(I)}\in B$ (or vice versa), then (assuming without loss of generality that $x_{0},x_{i'}\in A$ and $x_{\max(I)},x_{i'+1}\in B$) (\ref{eq:diamA>diamB}) implies that

\begin{align*}
\diam(\{x_i\}_{i\in I})&= d(x_{0},x_{\max(I)})\\
    &\leq d(x_{0},x_{i'})+d(x_{i'},x_{i'+1})+d(x_{i'+1},x_{\max(I)})\\
    &\leq 2\diam(\{x_i\}_{i\in I_A})+\delta.
\end{align*}
Hence, it suffices to prove
\begin{equation*}
    \diam(\{x_i\}_{i\in I_A}) \leq Q(Q+2)\delta.
\end{equation*}
Of course, since $f\res_A$ is $Q$-light, this will follow if we can prove that $\{x_i\}_{i\in I_A}$ is a $(Q+2)\delta$-chain. But this is easy to see: suppose $x_{j},x_{j'}$ are consecutive points in $(x_i)_{i\in I_A}$. If $j'=j+1$, then $d(x_j,x_{j'})\leq \delta$. If $j'\geq j+2$, then the points $\{x_i\}_{i=j+1}^{j'-1}$ form a $\delta$-chain in $B$. Therefore, the $Q$-lightness of $f\res_B$ implies that $\diam(\{x_i\}_{i=j+1}^{j'-1}) \leq Q\delta$. We conclude that
\begin{align*}
    d(x_j,x_{j'}) &\leq d(x_j,x_{j+1}) + d(x_{j+1},x_{j'-1}) + d(x_{j'-1},x_{j'})\\
    &\leq \delta + Q\delta + \delta\\
    &= (Q+2)\delta.
\end{align*}
\end{proof}

We now proceed to focus more specifically on QC trees and arcs, relying heavily on results from \cite{FG23} and \cite{Freeman20}. We will need the following technical lemmas.

\begin{lemma}\label{L:nice_ll_map}
Given a $C$-bounded turning Jordan arc $\gamma$, $r\geq0$, and $\{a,b\}\subset \mathbb{R}$ such that $0\leq|a-b|= r\cdot\diam(\gamma)$, there exists an $L$-Lipschitz $Q$-light map $f:\gamma\to\mathbb{R}$ such that $f$ maps the endpoints of $\gamma$ to $a$ and $b$. The constants $L$ and $Q$ depend only on $C$ and $\max\{1,r\}$.
\end{lemma}

\begin{proof}
This follows from the proof of \cite[Theorem 2.2]{Freeman20}. While the proof as written applies to bounded turning Jordan circles, the same construction can be applied to bounded turning Jordan arcs. 

By post-composing with a translation, it suffices to assume that $a=0$. Via the construction from \cite{Freeman20} (modified as indicated above), there exists an $L'$-Lipschitz $Q'$-light map $f: \gamma \to \R$, where $L',Q'$ depend only on $C$, such that the endpoints of $\gamma$ map to $\{0,\diam(\gamma)\}$. First suppose $r \geq 1$. Then the map $r \cdot f$ is an $L$-Lipschitz $Q$-light map, where $L,Q$ depend only on $C$ and $r$, sending the endpoints of $\gamma$ to $\{0,r\cdot\diam(\gamma)\} = \{a,b\}$, which proves the lemma in this case. Now assume that $0 \leq r \leq 1$. Then we post-compose $f$ with the 1-Lipschitz 3-light map
$$x \mapsto \begin{cases} x & x \leq \dfrac{r+1}{2}\diam(\gamma) \\ (r+1)\diam(\gamma)-x & x \geq \dfrac{r+1}{2}\diam(\gamma)
\end{cases},$$
producing an $L$-Lipschitz $Q$-light map, where $L,Q$ depend only on $C$, sending the endpoints of $\gamma$ to $\{0,r\cdot\diam(\gamma)\} = \{a,b\}$.
\end{proof}

\begin{lemma}\label{L:retract_nbrhd}
Given a $\delta$-chain $\{z_k\}_{k\in K}$ in a $1$-bounded turning tree $T$, every point of the arc $[z_0,z_{\max(K)}]$ is within distance $\delta$ of the set $\{z_k\}_{k\in K}$. 
\end{lemma}

\begin{proof}
Let $w\in [z_0,z_{\max(K)}]$ be fixed. Write $g:T\to[z_0,z_{\max(K)}]$ to denote the $1$-Lipschitz retraction defined in the proof of Lemma \ref{L:retract}. If $z_0=z_{\max(K)}$ then $w=z_{\max(K)}\in\{z_k\}_{k\in K}$. If $z_0\not=z_{\max(K)}$, then orient $[z_0,z_{\max(K)}]$ from $z_0$ to $z_{\max(K)}$.  Choose $k_0\in K$ to be the largest index such that $g(z_k)\leq w$ in $[z_0,z_{\max(K)}]$ for all $k\leq k_0$. If $k_0=\max(K)$, then \[w=g(z_{\max(K)})=z_{\max(K)}\in\{z_k\}_{k\in K}.\] Suppose $k_0<\max(K)$. Then $w\in[g(z_{k_0}),g(z_{k_0+1})]$. As argued in the proof of Lemma \ref{L:retract}, the fact that $g(z_{k_0})\not=g(z_{k_0+1})$ implies that $[g(z_{k_0}),g(z_{k_0+1})]\subset [z_{k_0},z_{k_0+1}]$. In particular, $w\in [z_{k_0},z_{k_0+1}]$. Therefore, the assumption that $T$ is $1$-bounded turning implies that
\[d(w,z_{k_0})\leq \diam[z_{k_0},z_{k_0+1}]=d(z_{k_0},z_{k_0+1})\leq \delta\] 
In conclusion, whether or not $z_0=z_{\max(K)}$, we have $d(w,\{z_k\}_{k\in K})\leq \delta$.
\end{proof}

The following lemma is a version of Lemma \ref{L:two_piece} that is tailored to the geometry of a bounded turning tree. 

\begin{lemma}\label{L:subset_and_components}
Suppose $T$ is a $C$-bounded turning QC tree, $X\subset T$ is closed. If there exists a map $F:T\to\mathbb{R}$ that is $L_0$-Lipschitz $Q_0$-light when restricted to $X$ or the closure of any component of $T\setminus X$, then $F:T\to\mathbb{R}$ is $L$-Lipschitz $Q$-light. Here, $L$ and $Q$ are determined only by $C$, $L_0$, and $Q_0$. 
\end{lemma}

\begin{proof}
By Lemma \ref{L:1bt}, we may assume that $T$ is $1$-bounded turning. Write $\{U_i\}_{i\in I}$ to denote the countably-many connected components of $T\setminus X$ (see Lemma \ref{L:components}). To see that $F$ is $L$-Lipschitz (for some $L$ depending only on $C$ and $L_0$), let $x,y\in T$. If both $x$ and $y$ are in $X$, or both are in a single $U_i$, then this is clear. Suppose $x\in X$ and $y\in U_i$ (for some $i\in I$). By Lemma \ref{L:components}(iv), there exists a point $z\in [x,y]\cap (X\cap\overline{U}_i)$. Then, since $T$ is $1$-bounded turning, we obtain 
\[|F(x)-F(y)|\leq |F(x)-F(z)|+|F(z)-F(y)|\leq 2L_0d(x,y),\]
Finally, suppose $x \in U_i$ and $y \in U_j$ for some $i\neq j \in I$. By Lemma \ref{L:components}(iv), there exist $z_i\in [x,y]\cap (X\cap\overline{U}_i)$ and $z_j\in [x,y]\cap (X\cap\overline{U}_j)$. Then, since $T$ is $1$-bounded turning, we obtain 
\[|F(x)-F(y)|\leq |F(x)-F(z_i)|+|F(z_i)-F(z_j)|+|F(z_j)-F(y)|\leq 3L_0d(x,y).\]

Thus $F:T\to\mathbb{R}$ is $L$-Lipschitz with $L=3L_0$.

To see that $F$ is $Q$-light (for some $Q$ depending only on $C$, $L_0$, and $Q_0$), fix $\delta>0$ and let $E\subset \mathbb{R}$ be such that $\diam(E)\leq \delta$. Let $\{x_j\}_{j\in J}$ denote a $\delta$-chain in $F^{-1}(E)$. For our purposes, we may assume that 
\[\diam(\{x_j\}_{j\in J})=d(x_0,x_{\max(J)}).\] 
Furthermore, by Lemma \ref{L:retract}, we may also assume $\{x_j\}_{j\in J}\subset[x_0,x_{\max(J)}]$. Indeed, since $T$ is $1$-bounded turning, we have 
\[\diam([x_0,x_{\max(J)}])=d(x_0,x_{\max(J)})=\diam(\{x_j\}_{j\in J}).\] 
In order to simplify our argument below, we may also assume (without affecting its diameter) that $\{x_j\}_{j\in J}$ proceeds monotonically from $x_0$ to $x_{\max(J)}$ along $[x_0,x_{\max(J)}]$. Since $F$ is $L$-Lipschitz, Lemma \ref{L:retract_nbrhd} implies that $[x_0,x_{\max(J)}]\subset F^{-1}(E')$, where $E'\subset \mathbb{R}$ is the set of points within distance $L\delta$ of $E$. 

If $[x_0,x_{\max(J)}]$ is contained in $X$, or in the closure of a single $U_i$, then we are done (see \cite[Lemma 5.9]{FG23}). Assume this is not the case, and let $\{(a_k,b_k)\}_{k\in K}$ denote the finitely-many components of $(x_0,x_{\max(J)})\setminus X$ containing at least one point of $\{x_j\}_{j\in J}$. Note that, by assumption, 
\begin{equation}\label{E:nontrivial}
(x_0,x_{\max(J)})\not\in\{(a_k,b_k)\}_{k\in K}\not=\emptyset.
\end{equation}
For each $k\in K$, write $U_{i_k}$ to denote the component of $T\setminus X$ containing $(a_k,b_k)$. Since $F$ is $Q_0$-light on each closure $\overline{U}_{i_k}$, it follows that 
\[d(a_k,b_k)=\diam([a_k,b_k])\leq 3LQ_0\delta.\]
Here we use the facts that $[a_k,b_k]\subset F^{-1}(E')$ and $\diam(E')\leq 3L\delta$. 

Therefore, by replacing each sub-chain $\{x_j\}_{j\in J}\cap (a_k,b_k)$ with the points $\{a_k,b_k\}$, we obtain a $3LQ_0\delta$-chain $\{z_j\}_{j\in J'}\subset[x_0,x_{\max(J)}]$ from $x_0$ to $x_{\max(J)}$. We note that, by construction, all points of $\{z_j\}_{j\in J'}$ are contained in $X$, except possibly $z_0=x_0$ and/or $z_{\max(J')}=x_{\max(J)}$. Since $F$ is $Q_0$-light on $X$, $\{z_j\}_{j\in J'}\subset F^{-1}(E')$, and $\diam(E')\leq 3LQ_0\delta$, we conclude that
\begin{align*}
\diam(\{x_j\}_{j\in J})=d(x_0,x_{\max(J)})&=\diam(\{z_j\}_{j\in J'})\\
    &\leq \diam(\{z_1,\dots,z_{\max(J')-1}\})+6LQ_0\delta\\
    &\leq 3LQ_0^2\delta+6LQ_0\delta\\
    &\leq9LQ_0^2\delta.
\end{align*}
Here we are also using the fact that $|J'|\geq2$, which follows from (\ref{E:nontrivial}). Set $Q=9LQ_0^2$, and it follows that $F:T\to\mathbb{R}$ is $Q$-light. 
\end{proof}

\begin{lemma}\label{L:leaf_extension}
Suppose $T$ is a $(C,D)$-QC tree, and $\L(T)$ is closed in $T$. If there exists an $L_0$-Lipschitz $Q_0$-light map $f:\L(T)\to\mathbb{R}$, then $f$ extends to an $L$-Lipschitz $Q$-light map $F:T\to\mathbb{R}$, where $L$ and $Q$ depend only on $C$, $D$, $L_0$, and $Q_0$. 
\end{lemma}

\begin{proof}
By Lemma \ref{L:1bt}, we may assume that $T$ is $1$-bounded turning. By Lemma~\ref{lem:Nagatadecomp}, there exists an $L''$-Lipschitz extension of $f$ to $\mathcal{L}(T)\cup \mathcal{B}(T)$, where $L''$ depends only on $D$ and $L_0$. Denote this extension by $G:(\L(T)\cup\mathcal{B}(T))\to\mathbb{R}$. We claim that $G$ is $Q''$-light, for $Q''$ determined by $C$, $D$, $L_0$, and $Q_0$. 

Before verifying our claim, we first explain how it implies the conclusion of the lemma. To this end, we first note that $T\setminus(\L(T)\cup\mathcal{B}(T))$ consists of countably many pairwise disjoint arcs. The closures of these arcs, denoted by $\{\gamma_j\}_{j\in J}$, are $(1,D)$-QC arcs whose endpoints $\{p_j,q_j\}_{j\in J}$ are contained in $\L(T)\cup \mathcal{B}(T)$. For each $j\in J$, write $a_j:=G(p_j)$ and $b_j:=G(q_j)$. Since $G$ is $L''$-Lipschitz on $(\L(T)\cup\mathcal{B}(T))$, Lemma \ref{L:nice_ll_map} implies the existence of an $L'$-Lipschitz $Q'$-light map $F_j:\gamma_j\to\mathbb{R}$ such that $(F_j)\res_{\{p_j,q_j\}}=G\res_{\{p_j,q_j\}}$. Here $L'$ and $Q'$ depend only on $C$, $D$, and $L_0$ (and not on $j$). Thus we define a map $F:T\to\mathbb{R}$ such that, for each $j\in J$, we have $F\res_{\gamma_j}=F_j$. Furthermore, $F\res_{(\L(T)\cap\mathcal{B}(T))}=G$. Via Lemma \ref{L:subset_and_components}, our claim implies that $F$ is $L$-Lipschitz $Q$-light, with $L$ and $Q$ determined only by $C$, $D$, $L_0$, and $Q_0$. 

\medskip
To verify our claim (and thus prove the lemma), let $\delta>0$ be fixed, and choose any $E\subset \mathbb{R}$ such that $\diam(E)\leq \delta$. Let $\{z_k\}_{k\in K}$ denote any $\delta$-chain in $G^{-1}(E)\subset \mathcal{L}(T)\cup \mathcal{B}(T)$. For the purpose of bounding the diameter of $\{z_k\}_{k\in K}$, we may assume that $\diam(\{z_k\}_{k\in K})=d(z_0,z_{\max(K)})$. Indeed, if $\diam(\{z_k\}_{k\in K})=d(z_n,z_m)$ for some $n<m\in K$, then we simply replace $\{z_0,\dots,z_{\max(K)}\}$ (if necessary) with $\{z_n,\dots,z_m\}$. 

\medskip
Case 1: No point of $\{z_k\}_{k\in K}$ is within distance $D^4\delta$ of $\mathcal{L}(T)$. We may (and do) assume that $D\geq 2$. Thus $\{z_k\}_{k\in K}$ consists entirely of branch points in $T$. By Lemma \ref{L:retract}, we obtain a $\delta$-chain $\{z_k'\}_{k\in K}\subset[z_0,z_{\max(K)}]$, again consisting solely of branch points in $T$. Here $z_0'=z_0$ and $z_{\max(K)}'=z_{\max(K)}$. By Lemma \ref{L:retract_nbrhd}, every point of $[z_0,z_{\max(K)}]$ is within distance $\delta$ of some point in $\{z_k\}_{k\in K}$. Therefore, 
\begin{equation}\label{E:distance_from_leaves}
d([z_0,z_{\max(K)}],\L(T))>(D^4-1)\delta\geq \delta.
\end{equation}

By way of contradiction, assume that $d(z_0,z_{\max(K)})\geq D^4\delta$. Let $u$ denote the first point of $[z_0,z_{\max(K)}]$ (moving from $z_0$ to $z_{\max(K)}$) such that $d(z_0,u)=D^4\delta$. Since $\{z_k'\}_{k\in K}$ is a $\delta$-chain, the arc $[z_0,u]$ contains at least $D^4$ distinct points from $\{z_k'\}_{k\in K}$. In particular, $[z_0,u]$ contains at least $D^4$ branch points of $T$. By (\ref{E:distance_from_leaves}), Lemma \ref{L:sep_pts}, and the fact that $T$ is $1$-bounded turning, we obtain points $\{b_i\}_{i\in I}\subset T$ such that, for every $i,i'\in I$, we have 
\[(D^4-1)\delta\leq d(b_i,b_{i'})\leq 2(D^4-1)\delta+d(z_0,u)\leq 4(D^4-1)\delta.\]
Furthermore, $|I|\geq D^4$. By the fact that $T$ is $D$-doubling, we also have $|I|\leq D^3$, a contradiction. We conclude that, if no point of $\{z_k\}_{k\in K}$ is within distance $D^4\delta$ to $\mathcal{L}(T)$, then $\diam(\{z_k\}_{k\in K})<D^4\delta$. 

\medskip
Case 2: At least one point of $\{z_k\}_{k\in K}$ is within distance $D^4\delta$ of $\L(T)$. If all points are within distance $D^4\delta$ of $\L(T)$, then we may skip to the next paragraph. If not, let $K'\subset K$ be a maximal (with respect to inclusion) sequence of consecutive indices such that no point of $\{z_k\}_{k\in K'}$ is within distance $D^4\delta$ of $\L(T)$. By Case 1, 
\begin{equation}\label{E:far_away_part}
\diam(\{z_k\}_{k\in K'})\leq D^4\delta.
\end{equation}
Since $K'\not=K$, some point of the $\delta$-chain $\{z_k\}_{k\in K'}$ is within distance $(D^4+1)\delta$ of $\L(T)$. Therefore, (\ref{E:far_away_part}) implies that every point of $\{z_k\}_{k\in K'}$ is within distance $(1+2D^4)\delta$ of $\L(T)$. Since this holds for any such $K'\subset K$, we conclude that all points of $\{z_k\}_{k\in K}$ are within distance $(1+2D^4)\delta$ of $\L(T)$. 

We write $D':=(1+2D^4)$. Via nearest point projections we obtain a sequence $\{w_k\}_{k\in K}\subset \L(T)$. It is clear that $\{w_k\}_{k\in K}$ is a $(1+2D')\delta$-chain. Furthermore, since $G$ is $L''$-Lipschitz, for each $k\in K$ we have $|G(z_k)-G(w_k)|\leq L''d(z_k,w_k)\leq L''D'\delta$. Thus $\{w_k\}_{k\in K}\subset G^{-1}(E')$, where $E'$ is defined to be the set of points at most distance $L''D'\delta$ from $E$. Note that 
\begin{equation}\label{E:bigger_set}
\diam(E')\leq \diam(E)+2L''D'\delta\leq(1+2L''D')\delta.
\end{equation}
Since $G\res_{\L(T)}=f$, $f$ is $Q_0$-light, and $(1+2D')\delta\leq (1+2L''D')\delta$, we conclude that 
\[\diam(\{w_k\}_{k\in K})\leq Q_0(1+2L''D')\delta.\]
It follows that, 
\[\diam(\{z_k\}_{k\in K})\leq Q_0(1+2D'+2L''D')\delta\]
in the case that some point of $\{z_k\}_{k\in K}$ is within distance $D^4\delta$ of $\L(T)$.

\medskip
Setting $Q:=Q_0(1+2D'+2L''D')$, the conjunction of Cases 1 and 2 confirm that any $\delta$-chain $\{z_k\}_{k\in K}\subset G^{-1}(E)$ has diameter at most $Q\delta$. Moreover, $Q$ is determined solely by $L_0$, $Q_0$, $C$, and $D$.
\end{proof}

\begin{lemma}\label{L:bdry_subset_extension}
Suppose $T$ is a $(C,D)$-QC tree, and $M\subset \mathcal{L}(T)$ is closed in $T$. Then every $L_0$-Lipschitz $Q_0$-light map $f:M\to\mathbb{R}$ extends to an $L$-Lipschitz $Q$-light map $F:T\to\mathbb{R}$. Here $L$ and $Q$ depend only on $L_0$, $Q_0$, $C$, and $D$.
\end{lemma}

\begin{proof}
Set $S:=\hull(M)\subset T$. By Lemmas~\ref{L:closed_hull} and \ref{L:leaf_extension}, there exists an $L''$-Lipschitz $Q''$-light map $G:S\to\mathbb{R}$ such that $G\res_M=f$, where $L''$ and $Q''$ depend only on $C$, $D$, $L_0$ and $Q_0$. Let $\{U_i\}_{i\in I}$ denote the connected components of $T\setminus S$. Write $\{T_i\}_{i\in I}$ to denote their closures. By Lemma \ref{L:components}, each $T_i$ is a sub-tree of $T$. By \cite[Theorem 10.10]{Nadler92}, each intersection $V_i:=T_i\cap S$ is connected, and thus also a subtree of $T$. It follows from the connectedness of $U_i$ that $V_i\subset\mathcal{L}(T_i)$, and so the subtree $V_i$ must be degenerate. That is, $V_i$ consists of a single point. In this way we define, for each $i\in I$, the point $p_i\in T_i\cap S$. By \cite[Theorem 5.10]{FG23}, for each $i\in I$, there exists an $L'$-Lipschitz $Q'$-light map $F_i:T_i\to\mathbb{R}$, where $L'$ and $Q'$ depend only on $C$ and $D$. Up to translating images in $\mathbb{R}$, we may assume that $F_i(p_i)=G(p_i)$. Thus we obtain a continuous map $F:T\to\mathbb{R}$ such that $F\res_S=G$ and, for every $i\in I$, $F\res_{T_i}=F_i$. By Lemma \ref{L:subset_and_components}, the map $F:T\to\mathbb{R}$ is $L$-Lipschitz $Q$-light, for $L$ and $Q$ determined solely by $C$, $D$, $L_0$, and $Q_0$. 
\end{proof}

\begin{proof}[Proof of Theorem \ref{T:Lip_dim_unions}]
Let $T_1,\dots T_k \subset Z$ be a finite collection of QC trees inside an ambient metric space $Z$. We will prove that $\bigcup_{i=1}^k T_i$ has Lipschitz dimension 1 by induction on $k$. By \cite[Theorem D]{FG23}, the conclusion holds when $k=1$. Now let $k \geq 2$, set $X := \bigcup_{i=1}^{k-1} T_i$ and $T := T_k$, and assume that $X$ has Lipschitz dimension 1. Let $f: X \to \R$ be a Lipschitz light map.

Let $N$ be a $1/2$-Whitney net in $T$ with respect to $X$, and let $\pi: X \cup N \to X$ be any map with $d(u,\pi(u)) \leq 2d(u,X)$ for all $u \in X \cup N$. Note that this implies $\pi\res_X$ is the identity map on $X$. By Lemma \ref{L:light_projection}, the map $\pi$ is Lipschitz light. Then $f \circ \pi: X \cup N \to \R$ is Lipschitz light, being the composition of Lipschitz light maps.

Set $Y := T \cap (X \cup N) = (T \cap X) \cup N$. It is easy to see that $Y$ is a closed subset of $T$. Since $Y$ is closed in $T$, the closures of the countably-many components of $T \setminus Y$, say $\{S_i\}_{i\in I}$, are subtrees of $T$ (see Lemma \ref{L:components}). For each $i\in I$, set $M_{i} := S_{i} \cap Y \subset \mathcal{L}(S_i)$. By Lemma \ref{L:bdry_subset_extension}, for each $i \in I$, the map $(f \circ \pi)\res_{M_i}$ extends to an $L$-Lipschitz $Q$-light map $f_i: S_i \to \R$, where $L,Q<\infty$ are independent of $i$. Gluing these $f_i$ and $f \circ \pi$ together, we obtain a well-defined map $F: T \cup X \to \R$ satisfying $F\res_{S_i} = f_i$ and $F\res_{X \cup N} = f \circ \pi$. By Lemma \ref{L:subset_and_components}, the restriction $F\res_T$ is Lipschitz light, and by Lemma \ref{L:two_piece}, the entire map $F$ is $Q'$-light for some $Q' < \infty$.

It remains to verify that $F$ is Lipschitz. To this end, let $x,y \in T \cup X$. Without loss of generality, there are three cases to consider: $x,y \in T$, $x \in T \setminus X$ and $y \in X$, or $x,y \in X$. The Lipschitz inequalities for the first and third cases hold since $F\res_T$ and $F\res_X$ are Lipschitz. Thus, we proceed to consider the second case. By Lemma~\ref{L:Whitney}, there exists $u \in N$ such that
\begin{equation} \label{eq:ineq1}
    d(x,u) \leq d(x,X) \leq d(x,y)
\end{equation}
Let $L'$ be the maximum of the Lipschitz constants of $F\res_T$ and $F\res_{X \cup N}$. Then since $x,u \in T$ and $y,u \in X \cup N$, we have
\begin{align*}
    |F(x)-F(y)| &\leq |F(x)-F(u)| + |F(u)-F(y)| \\
    &\leq L'd(x,u) + L'd(u,y) \\
    &\leq L'd(x,u) + L'd(u,x) + L'd(x,y) \\
    &\overset{\eqref{eq:ineq1}}{\leq} 3L'd(x,y),
\end{align*}
completing the proof.
\end{proof}

%%==================================================%%
\subsection{Proof of Theorem \ref{T:quotient_dim}}\label{s:quotient_dim}
%%==================================================%%

We bootstrap our way towards a proof of Theorem \ref{T:quotient_dim}. First, we prove that the Lipschitz dimension of a QC wreath is equal to one. Then, via a series of lemmas, we reduce the general case to a consideration of QC wreaths and QC trees. 

\begin{lemma}\label{L:wreath_dim}
If $X$ is a QC wreath obtained from a $(C,D$)-QC tree $T$, then there exists an $L$-Lipschitz $Q$-light map $g:X\to\mathbb{R}$ such that $L$ and $Q$ depend only on $C$ and $D$.  
\end{lemma}

\begin{proof}
We may write $X=T/P$, where $P=\{a,b\}\subset \mathcal{L}(T)$. Let $\rho$ denote the quotient metric on $T/P$. By \cite[Theorem D]{FG23} (see also \cite[Theorem 5.10]{FG23}), there exists an $L$-Lipschitz $Q'$-light map $f:T\to\mathbb{R}$, where $L$ and $Q'$ depend only on $C$ and $D$. Then we post-compose $f$ with the 1-Lipschitz 3-light map
$$x \mapsto \begin{cases} x & x \leq \dfrac{f(a)+f(b)}{2} \\ f(a)+f(b)-x & x \geq \dfrac{f(a)+f(b)}{2}
\end{cases},$$
obtaining an $L$-Lipschitz $Q''$-light map $g: T \to \R$ with $g(a)=g(b)$, where $L,Q''$ depend only on $C,D$. Then $g$ descends to a map $\bar{g}: T/P \to \R$, defined by the property $g = \bar{g} \circ \pi$, where $\pi: T \to T/P$ is the projection. It follows from Remark~\ref{R:quotient} that $\Lip(\bar{g}) = \Lip(g) = L$.

To prove $Q$-lightness of $\bar{g}$, we need the following claim.

\medskip
\noindent \underline{Claim}: \emph{For all $\delta \geq 0$ and all $\delta$-chains $\{[z_k]\}_{k\in K} \subset T/P$ with $\rho(P,\{[z_k]\}_{k\in K}) \leq \delta$, we have}
\begin{equation*} \label{eq:pi^-1(z_k)}
    \pi^{-1}(\{[z_k]\}_{k\in K}) = \widetilde{E}_a \cup \widetilde{E}_b,
\end{equation*}
\emph{where $\widetilde{E}_a,\widetilde{E}_b$ are $2\delta$-connected sets with $d(a,\widetilde{E}_a) \leq \delta$ and $d(b,\widetilde{E}_b) \leq \delta$.}
\medskip

We prove this by induction on $|K|$. The base case $|K| = 1$ is straightforward. Now suppose that $|K| \geq 2$. By the inductive hypothesis, we have $\pi^{-1}(\{[z_k]\}_{k < \max(K)}) = \widetilde{E}_a \cup \widetilde{E}_b$ for some $2\delta$-connected sets $\widetilde{E}_a,\widetilde{E}_b$ with $d(a,\widetilde{E}_a) \leq \delta$ and $d(b,\widetilde{E}_b) \leq \delta$. Since $\{[z_k]\}_{k\in K}$ is a $\delta$-chain, we have that $\rho([z_{\max(K)-1}],[z_{\max(K)}]) \leq \delta$. Let $u \in \pi^{-1}([z_{\max(K)-1}])$ and $v \in \pi^{-1}([z_{\max(K)}])$. If $[z_{\max(K)}] = P$, then $\pi^{-1}(\{[z_k]\}_{k \in K}) = (\widetilde{E}_a \cup \{a\}) \cup (\widetilde{E}_b \cup \{b\})$ which satisfies the desired conclusion. Hence, we may assume that $\{v\} = \pi^{-1}([z_{\max(K)}])$. We have
$$\delta \geq \rho([z_{\max(K)-1}],[z_{\max(K)}]) = \min\{d(u,v),d(u,P)+d(v,P)\},$$
which implies that (i) $d(u,v) \leq \delta$, (ii) $d(v,b)\leq \delta$, or (iii) $d(v,a)\leq \delta$. Suppose (i) holds. By the inductive hypothesis, $u \in \widetilde{E}_a \cup \widetilde{E}_b$, and without loss of generality we may assume that $u \in \widetilde{E}_a$. Then $\widetilde{E}_a \cup \{v\}$ is $2\delta$-connected and $\pi^{-1}(\{[z_k]\}_{k \in K}) = (\widetilde{E}_a \cup \{v\}) \cup \widetilde{E}_b$, proving the claim in this case. Now suppose (ii) holds. Then $d(v,b) \leq \delta$, and thus $\widetilde{E}_b \cup \{v\}$ is $2\delta$-connected and $\pi^{-1}(\{[z_k]\}_{k \in K}) = \widetilde{E}_a \cup (\widetilde{E}_b \cup \{v\})$, proving the claim in this case. The last case (iii) follows from a similar argument.

With this claim in hand, we can prove that $\bar{g}$ is $Q$-light. Let $\delta \geq 0$ and $E \subset \R$ with $\diam(E) \leq \delta$. Let $\{[z_k]\}_{k \in K} \subset \bar{g}^{-1}(E)$ be a $\delta$-chain. We consider two cases: (i) $\rho(P,\{[z_k]\}_{k \in K}) \geq \delta$ and (ii) $\rho(P,\{[z_k]\}_{k \in K}) \leq \delta$. Assume (i) holds. Then for all $k < \max(K)$, we have
$$\delta \geq \rho([z_k],[z_{k+1}]) = \min\{d(z_k,z_{k+1}),d(z_k,P)+d(z_{k+1},P)\}$$
$$\geq \min\{d(z_k,z_{k+1}),2\delta\},$$
implying that $\{z_k\}_{k \in K}$ is a $\delta$-chain in $g^{-1}(E)$. Since $g$ is $Q''$-light, this implies 
\[\diam(\{[z_k]\}_{k \in K}) \leq \diam(\{z_k\}_{k\in K}) \leq Q''\delta,\]
proving $Q''$-lightness in this case. Now suppose (ii) holds. By the claim, $\pi^{-1}(\{[z_k]\}_{k\in K}) = \widetilde{E}_a \cup \widetilde{E}_b$, where $\widetilde{E}_a,\widetilde{E}_b$ are $2\delta$-connected sets with $d(a,\widetilde{E}_a) \leq \delta$ and $d(b,\widetilde{E}_b) \leq \delta$. Then since $\widetilde{E}_a,\widetilde{E}_b \subset g^{-1}(E)$ and $g$ is $Q''$-light, we have that $\diam(\widetilde{E}_a), \diam(\widetilde{E}_b) \leq 2Q''\delta$. This inequality together with the fact that $d(a,\widetilde{E}_a),d(b,\widetilde{E}_b) \leq \delta$ implies
$$\diam(\{[z_k]\}_{k\in K}) = \diam(\pi(\widetilde{E}_a \cup \widetilde{E}_b)) \leq 4Q''\delta+2\delta,$$
proving $(4Q''+2)$-lightness in this final case.
\end{proof}

We will also use the following, which may be of independent interest as it provides a new characterization of uniform disconnectedness. 

\begin{proposition}\label{P:disconnected_quotient}
Let $\alpha\in (0,1)$ and suppose $X$ is a metric space. If $Y$ is an $\alpha$-uniformly disconnected subset of $X$, then the quotient map $\pi:X\to X/\overline{Y}$ is $1$-Lipschitz $Q$-light, for $Q=(9/\alpha+8)$. Conversely, if $\pi:X\to X/\overline{Y}$ is $1$-Lipschitz $Q$-light, then $Y$ is $\beta$-uniformly disconnected for any $\beta<1/Q$. 
\end{proposition}

\begin{proof}
We first assume that $Y$ is $\alpha$-uniformly disconnected. Note that $\pi$ is clearly $1$-Lipschitz (regardless of any assumptions on $Y$). To show that $\pi$ is $Q$-light, we begin with the following claim.

\medskip
Claim: For every $\eps>0$, each $\eps$-component of $Y$ has diameter no greater than $\eps/\alpha$. 
\medskip

To verify this first claim, let $a,b$ be arbitrary points in an $\eps$-component $C$ of $Y$, and let $\{z_k\}_{k\in K}$ be an $\eps$-chain in $Y$ connecting $a$ to $b$. Since $Y$ is $\alpha$-uniformly disconnected, there must exist some $1\leq k_0\leq \max(K)$ such that 
\[\varepsilon\geq d(z_{k_0-1},z_{k_0})>\alpha\, d(a,b),\]
showing $d(a,b) < \eps/\alpha$. Since $a,b \in C$ we arbitrary, this shows $\diam(C) \leq \eps/\alpha$.

\medskip
Next, let $\delta>0$ be given, and let $[E]\subset X/\overline{Y}$ be such that $\diam([E])\leq \delta$. Let $U$ denote a $\delta$-component of $\pi^{-1}([E])$. We consider two cases for $U$.

\medskip
Case 1: There exists a point $z\in U$ such that $d(z,Y)\geq 4\delta$.
\medskip

In this case, we first note that for any points $x,y$ within distance $2\delta$ of $z$, we have $\rho([x],[y])=d(x,y)$. Indeed, this follows immediately from the definition of the quotient distance $\rho$. Thus $\pi$ restricted to $B:=B(z;2\delta)$ is an isometry, and the diameter of any $\delta$-chain $\{z_k\}_{k\in K}$ in $\pi^{-1}([E])\cap B$ is preserved by $\pi$. Since $\pi(\{z_k\}_{k\in K})\subset [E]$ and $\diam([E])\leq \delta$, we note that:
\begin{equation}\label{E:small_diam}
\textrm{Any } \delta\textrm{-chain in } \pi^{-1}([E])\cap B \textrm{ has diameter at most } \delta.
\end{equation}

Let $u,v \in U$ be arbitrary, and let $\{w_j\}_{j\in J}$ be a $\delta$-chain in $U$ containing $u,v,z$, which exists since $u,v,z \in U$ and $U$ is a $\delta$-component of $\pi^{-1}([E])$. Suppose, by way of contradiction, that $\{w_j\}_{j\in J}\not\subset B$. By re-indexing, we may assume that $w_{j_0}=z$ and $\{z_{j_0},\dots,z_{\max(J)}\}\not\subset B$ for some $j_0<\max(J)$. Write $j_1$ to denote the first index after $j_0$ such that $d(w_{j_1},z)> 2\delta$. Since $\{w_j\}_{j\in J}$ is a $\delta$-chain, the triangle inequality implies that $d(w_{j_1-1},w_{j_0})>\delta$. Since $\{w_{j_0},\dots,w_{j_1-1}\}$ is a $\delta$-chain in $\pi^{-1}([E])\cap B$, we contradict (\ref{E:small_diam}). Thus $\{w_j\}_{j\in J}\subset B$. By (\ref{E:small_diam}), we conclude that $d(u,v) \leq \delta$. Since $u,v \in U$ were arbitrary, this implies that, in Case 1, we have $\diam(U)\leq \delta$.

\medskip
Case 2: For every point $z\in U$, we have $d(z,Y)<4\delta$.
\medskip

In this case, let $\{z_k\}_{k\in K}$ denote any $\delta$-chain in $U$. By assumption, each $z_k$ is within distance $4\delta$ of a point $z_k'\in Y$. Therefore, we induce a $9\delta$-chain $\{z_k'\}_{k\in K}$ consisting of points in $Y$. By the Claim above, we have $\diam(\{z_k'\}_{k\in K})\leq 9\delta/\alpha$. This implies that, in Case 2, we have $\diam(U)\leq (8+9/\alpha)\delta$.

\medskip
Combining Cases 1 and 2, we conclude that any $\delta$-component of $\pi^{-1}([E])$ has diameter at most $(9/\alpha+8)\delta$. It follows that $\pi$ is $(9/\alpha+8)$-light. 

To conclude the proof of the lemma, we assume that $\pi:X\to X/\overline{Y}$ is $1$-Lipschitz $Q$-light. Fix any two points $x,y\in Y$. Let $\delta>0=\diam([\overline{Y}])$ be given, and let $\{z_k\}_{k\in K}$ denote a $\delta$-chain from $x$ to $y$ in $Y\subset\pi^{-1}([\overline{Y}])$. By assumption, $\diam(\{z_k\}_{k\in K})\leq Q\delta$. Therefore, $\delta\geq d(x,y)/Q$. The desired conclusion follows. 
\end{proof}

\begin{lemma}\label{L:sum_lipdim}
Suppose $X=\coprod_{i\in I}X_i$. If there exist uniformly $L$-Lipschitz $Q'$-light maps $f_i:X_i\to\mathbb{R}$, then there exists an $L$-Lipschitz $Q$-light map $f:X\to \mathbb{R}$, where $Q$ depends only on $L$ and $Q'$. 
\end{lemma}

\begin{proof}
This lemma is merely a special case of \cite[Theorem F]{FG23} (see also \cite[Theorem 4.4]{FG23}), since the spaces $\{X_i\}_{i\in I}$ form a $2$-geometric tree-like decomposition of $X$.
\end{proof}

\begin{proof}[Proof of Theorem \ref{T:quotient_dim}]
Let $T$ denote a $(C,D)$-QC tree. By Lemma \ref{L:1bt}, we may assume that $T$ is $1$-bounded turning. By Lemma \ref{L:pre_coproduct}, we know that $T/M$ is $2$-bi-Lipschitz equivalent to $\coprod_{i\in I}T_i/M_i$, where $\{T_i\}_{i\in I}$ denotes the closures of the countably many connected components of $T\setminus M$, and $M_i:=T_i\cap M\subset \mathcal{L}(T_i)$. Let $S_i$ denote $\hull(M_i)\subset T_i$, and write $B_i$ to denote $\mathcal{B}(S_i)$. By Lemmas \ref{L:double_quotient} and \ref{L:coproduct}, each  $(T_i/M_i)/((M_i\cup B_i)/M_i)=T_i/(M_i\cup B_i)$ is $2$-bi-Lipschitz equivalent to $\coprod_{j\in J_i}X_{i,j}$, where each $X_{i,j}$ is either a $(1,D)$-QC tree or $(1,D)$-QC wreath. If $X_{i,j}$ is a tree, then by \cite[Theorem 5.10]{FG23}, there exists an $L$-Lipschitz $Q$-light map $f_{i,j}:X_{i,j}\to\mathbb{R}$ such that $L$ and $Q$ depend only on $C$ and $D$. If $X_{i,j}$ is a wreath, then, by Lemma \ref{L:wreath_dim}, there exists an $L_0$-Lipschitz $Q_0$-light map $f_{i,j}:X_{i,j}\to \mathbb{R}$ such that $L_0$ and $Q_0$ depend only on $C$ and $D$. By Lemma \ref{L:sum_lipdim}, for each $i\in I$, there exists an $L_1$-Lipschitz $Q_1$-light map $g_i:T_i/(M_i\cup B_i)\to\mathbb{R}$, where $L_1$ and $Q_1$ depend only on $C$ and $D$. Write $\pi$ to denote the quotient map $\pi:(T_i/M_i)\to T_i/(M_i\cup B_i)$ (here we again use Lemma \ref{L:double_quotient}). By Theorem \ref{T:uniform} and Proposition~\ref{P:disconnected_quotient}, the map $\pi$ is $1$-Lipschitz $Q_2$-light, with $Q_2$ depending only on $C$ and $D$. Thus, it is easy to check that the composition $g_i\circ \pi:T_i/M_i\to\mathbb{R}$ is Lipschitz light, with constants depending only on $C$, and $D$. Finally, again appealing to Lemma \ref{L:sum_lipdim}, there exists a Lipschitz light map $g:T/M\to \mathbb{R}$ with constants depending only on $C$ and $D$. In particular, $\dim_L(T/M)=1$. 
\end{proof}

\section*{Acknowledgements}
We are very grateful to the anonymous referee for their thorough reading of the original manuscript and for pointing out numerous mistakes, in particular for bringing to our attention a fundamental error in the first (faulty) proof of Theorem~\ref{T:Lip_dim_unions}.

\bibliography{LipFunctions_Rev2.bib}

\providecommand{\bysame}{\leavevmode\hbox to3em{\hrulefill}\thinspace}
\providecommand{\MR}{\relax\ifhmode\unskip\space\fi MR }
\providecommand{\MRhref}[2]{%
  \href{http://www.ams.org/mathscinet-getitem?mr=#1}{#2}
}
\providecommand{\href}[2]{#2}
\begin{thebibliography}{ACPCS01}

\bibitem[AACD21]{self-similar}
Fernando Albiac, Jos\'{e}~L. Ansorena, Marek C\'{u}th, and Michal Doucha,
  \emph{Lipschitz free spaces isomorphic to their infinite sums and geometric
  applications}, Trans. Amer. Math. Soc. \textbf{374} (2021), no.~10,
  7281--7312.

\bibitem[ACPCS01]{ACCS01}
Daniel Ar\'{e}valo, W\l odzimierz~J. Charatonik, Patricia Pellicer~Covarrubias,
  and Likin Sim\'{o}n, \emph{Dendrites with a closed set of end points},
  Topology Appl. \textbf{115} (2001), no.~1, 1--17. \MR{1840729}

\bibitem[AGPP22]{AGPP22}
Ram\'{o}n~J. Aliaga, Chris Gartland, Colin Petitjean, and Anton\'{\i}n
  Proch\'{a}zka, \emph{Purely 1-unrectifiable metric spaces and locally flat
  {L}ipschitz functions}, Trans. Amer. Math. Soc. \textbf{375} (2022), no.~5,
  3529--3567.

\bibitem[AO01]{AO}
Dale Alspach and Edward Odell, \emph{{$L_p$} spaces}, Handbook of the geometry
  of {B}anach spaces, {V}ol. {I}, North-Holland, Amsterdam, 2001, pp.~123--159.

\bibitem[BM20]{BM20a}
Mario Bonk and Daniel Meyer, \emph{Quasiconformal and geodesic trees}, Fund.
  Math. \textbf{250} (2020), no.~3, 253--299. \MR{4107537}

\bibitem[BM22]{BM22}
\bysame, \emph{Uniformly branching trees}, Trans. Amer. Math. Soc. \textbf{375}
  (2022), no.~6, 3841--3897. \MR{4419049}

\bibitem[BT21]{BT21}
Mario Bonk and Huy Tran, \emph{The continuum self-similar tree}, Fractal
  geometry and stochastics {VI}, Progr. Probab., vol.~76,
  Birkh\"{a}user/Springer, Cham, [2021] \copyright 2021, pp.~143--189.
  \MR{4237253}

\bibitem[CD16]{CD16}
Marek C\'{u}th and Michal Doucha, \emph{Lipschitz-free spaces over ultrametric
  spaces}, Mediterr. J. Math. \textbf{13} (2016), no.~4, 1893--1906.
  \MR{3530906}

\bibitem[CDW16]{CDW16}
Marek C\'{u}th, Michal Doucha, and Przemys\l~aw Wojtaszczyk, \emph{On the
  structure of {L}ipschitz-free spaces}, Proc. Amer. Math. Soc. \textbf{144}
  (2016), no.~9, 3833--3846. \MR{3513542}

\bibitem[CK13]{CK13}
Jeff Cheeger and Bruce Kleiner, \emph{Realization of metric spaces as inverse
  limits, and bilipschitz embedding in {$L_1$}}, Geom. Funct. Anal. \textbf{23}
  (2013), no.~1, 96--133. \MR{3037898}

\bibitem[Dav21]{David21}
Guy~C. David, \emph{On the {L}ipschitz dimension of {C}heeger-{K}leiner}, Fund.
  Math. \textbf{253} (2021), no.~3, 317--358. \MR{4205978}

\bibitem[DEBV23]{DEV23}
Guy David, Sylvester Eriksson-Bique, and Vyron Vellis, \emph{Bi-{L}ipschitz
  embeddings of quasiconformal trees}, Proc. Amer. Math. Soc. \textbf{151}
  (2023), no.~5, 2031----2044. \MR{4556198}

\bibitem[DS97]{DS97}
Guy David and Stephen Semmes, \emph{Fractured fractals and broken dreams},
  Oxford Lecture Series in Mathematics and its Applications, vol.~7, The
  Clarendon Press, Oxford University Press, New York, 1997, Self-similar
  geometry through metric and measure. \MR{1616732}

\bibitem[FG]{FG23}
David Freeman and Chris Gartland, \emph{Lipschitz functions on quasiconformal
  trees}, Fund. Math. (to appear).

\bibitem[Fre22]{Freeman20}
David~M. Freeman, \emph{Weak quasicircles have {L}ipschitz dimension 1}, Ann.
  Fenn. Math. \textbf{47} (2022), no.~1, 283--303. \MR{4366417}

\bibitem[GK03]{GK03}
G.~Godefroy and N.~J. Kalton, \emph{Lipschitz-free {B}anach spaces}, vol. 159,
  2003, Dedicated to Professor Aleksander Pe\l czy\'{n}ski on the occasion of
  his 70th birthday, pp.~121--141. \MR{2030906}

\bibitem[God10]{Godard10}
A.~Godard, \emph{Tree metrics and their {L}ipschitz-free spaces}, Proc. Amer.
  Math. Soc. \textbf{138} (2010), no.~12, 4311--4320. \MR{2680057}

\bibitem[JL01]{JL}
William~B. Johnson and Joram Lindenstrauss, \emph{Basic concepts in the
  geometry of {B}anach spaces}, Handbook of the geometry of {B}anach spaces,
  {V}ol. {I}, North-Holland, Amsterdam, 2001, pp.~1--84. \MR{1863689}

\bibitem[Kau14]{Kaufmann14}
Pedro~Levit Kaufmann, \emph{Products of {L}ipschitz-free spaces and
  applications}, arXiv:1403.6605 (2014).

\bibitem[Kau15]{Kaufmann15}
\bysame, \emph{Products of {L}ipschitz-free spaces and applications}, Studia
  Math. \textbf{226} (2015), no.~3, 213--227. \MR{3356002}

\bibitem[LD17]{primer}
Enrico Le~Donne, \emph{A primer on {C}arnot groups: homogenous groups,
  {C}arnot-{C}arath\'{e}odory spaces, and regularity of their isometries},
  Anal. Geom. Metr. Spaces \textbf{5} (2017), no.~1, 116--137.

\bibitem[LP01]{LP01}
Urs Lang and Conrad Plaut, \emph{Bilipschitz embeddings of metric spaces into
  space forms}, Geom. Dedicata \textbf{87} (2001), no.~1-3, 285--307.
  \MR{1866853}

\bibitem[LS05]{LS05}
Urs Lang and Thilo Schlichenmaier, \emph{Nagata dimension, quasisymmetric
  embeddings, and {L}ipschitz extensions}, Int. Math. Res. Not. (2005), no.~58,
  3625--3655. \MR{2200122}

\bibitem[Nad92]{Nadler92}
Sam~B. Nadler, Jr., \emph{Continuum theory}, Monographs and Textbooks in Pure
  and Applied Mathematics, vol. 158, Marcel Dekker, Inc., New York, 1992, An
  introduction. \MR{1192552}

\bibitem[NS11]{NaorSilb}
Assaf Naor and Lior Silberman, \emph{Poincar\'{e} inequalities, embeddings, and
  wild groups}, Compos. Math. \textbf{147} (2011), no.~5, 1546--1572.

\bibitem[Wea18]{Weaver18}
Nik Weaver, \emph{Lipschitz algebras}, World Scientific Publishing Co. Pte.
  Ltd., Hackensack, NJ, 2018, Second edition of [ MR1832645]. \MR{3792558}

\end{thebibliography}
\bibliographystyle{amsalpha}

\end{document}